\documentclass[11pt,a4paper,reqno]{amsart}
\usepackage{xcolor}
\usepackage{amsfonts}
\usepackage{amsmath}
\usepackage{amssymb}
\usepackage{hyperref}

\newcommand{\Div}{\operatorname{div}}

\newcommand{\defeq}{\mathrel{\mathop:}=}
\numberwithin{equation}{section}
\newtheorem{theorem}{Theorem}[section]
\newtheorem{lemma}[theorem]{Lemma}
\newtheorem{remark}[theorem]{Remark}
\newtheorem{corollary}[theorem]{Corollary}
\newtheorem{proposition}[theorem]{Proposition}
\newtheorem{definition}[theorem]{Definition}

\def\eqref#1{(\ref{#1})}
\allowdisplaybreaks

\newcommand{\norm}[1]{\left\lVert#1\right\rVert}

\def\enddoc{\end{document}}

\def\FRAME#1#2#3#4#5#6#7#8
{
\begin{figure}[ptbh]
	\begin{center}
		\includegraphics[height=#3]{#7}
		\caption{#5}
		\label{#6}
	\end{center}
\end{figure}
}

\begin{document}
\author{Qingsong Gu}
\address{Department of mathematics, Nanjing University, Nanjing 210093, P. R. China} \email{qingsonggu@nju.edu.cn}

\author{Xueping Huang}
\address{School of Mathematics and Statistics, Nanjing University of Information Science and Technology,
	Nanjing 210044, P. R. China}
\email{hxp@nuist.edu.cn}

\author{Yuhua Sun}
\address{School of Mathematical Sciences and LPMC, Nankai University, 300071
	Tianjin, P. R. China}
\email{sunyuhua@nankai.edu.cn}

\title{Sharp criteria for nonlocal elliptic inequalities on manifolds}
\thanks{\noindent Q. Gu  was supported by the National Natural Science Foundation of China (Grant Nos. 12101303 and 12171354). X. Huang was supported by
	the National Natural Science Foundation of China (Grant No. 11601238). Y. Sun was supported by the National Natural Science Foundation of
	China (Grant No.11501303).}

\begin{abstract}
	%
	Let $M$ be a complete non-compact Riemannian manifold and  $\sigma $ be a
	Radon measure on $M$, we study the existence and non-existence of
	positive solutions to a nonlocal elliptic inequality
	\begin{equation*}
		(-\Delta)^{\alpha} u\geq u^{q}\sigma\quad \text{in}\,\,M,
		\end{equation*}%
	with $q>1$. When the Green function
	$G^{(\alpha)}$ of the fractional Laplacian 	$(-\Delta)^{\alpha}$ exists and satisfies the quasi-metric property, we obtain necessary and sufficient criteria for existence of
	positive solutions. In particular,
	explicit conditions in terms of volume growth and the growth of $\sigma$ are given, when $M$ admits Li-Yau Gaussian type  heat kernel estimates.
\end{abstract}

\maketitle
\tableofcontents
\section{Introduction}

Let $\Delta$ be the Laplace-Beltrami operator on a complete connected non-compact Riemannian manifold $(M,g)$ and $\alpha\in (0,1)$ be a constant. Through spectral calculus or subordination theory, we can investigate the fractional Laplacian $(-\Delta)^{\alpha}$ (see \eqref{frac-lap} for the precise definition). We are interested in the existence of positive solutions to the following fractional differential inequality
\begin{equation}\label{eq-non-local-differential-inequality}
	(-\Delta)^{\alpha} u\geq u^{q} \sigma,\quad\mbox{on $M$},
\end{equation}%
with $ q>1$ being a constant and $\sigma$ being a Radon measure on $M$.
	
	Assume that the operator $(-\Delta)^{\alpha}$ is transient, namely there is an  associated Green kernel function $G^{(\alpha)}(\cdot , \cdot): M\times M \rightarrow (0, +\infty]$ which is 
	lower semi-continuous,  finite off the diagonal, and is formally the inverse of $(-\Delta)^{\alpha}$ (see Section \ref{sect-frac-Lap} for the precise definitions and technical details). It is more convenient to study the closely related integral inequality
	\begin{equation}\label{eq-non-local-integral-inequality}
		u(x)\geq \int_{M}G^{(\alpha)}(x,y)u^{q}(y) d \sigma (y). 
	\end{equation}

The solution of (\ref{eq-non-local-integral-inequality}) is understood in the following sense.
\begin{definition}\label{defi-int-ineq}\rm
		Let $u$ be a $\sigma$-a.e. defined function that admits a lower semi-continuous $\sigma$-version. Then $u$ is called a positive solution to (\ref{eq-non-local-integral-inequality}) if 	for $\sigma$-a.e. $x\in M$, $u(x)\in (0, +\infty)$, and (\ref{eq-non-local-integral-inequality})
		holds. 
	\end{definition}
\begin{remark}\rm
	Let $\tilde{u}$ be a lower semi-continuous $\sigma$-version of a positive solution $u$. Note that for $\sigma$-a.e. $x\in M$, $\tilde{u}(x)\in (0, +\infty)$, and (\ref{eq-non-local-integral-inequality})
	holds for $\tilde{u}$ in place of $u$. Later without further specification, we  simply work with the lower semi-continuous $\sigma$-version of a positive solution.
\end{remark}

Let $d$ be the geodesic distance and $\mu$ be the Riemannian measure on $M$.
Denote by $B(x,r)$ the geodesic ball centered at $x$ with radius $r$. The following two conditions are very important in geometric analysis.
\begin{itemize}
	\item[1.] The manifold $(M, g)$ is said to satisfy the volume doubling condition if for all $x\in M$ and $r>0$,
	\begin{equation}
		\mu (B(x,2r))\lesssim \mu (B(x,r));  \tag{VD}  \label{D}
	\end{equation}
	\item[2.] The manifold $(M, g)$ is said to satisfy the (scale invariant) Poincar\'{e} inequality if for any ball $B=B(x,r)\subset M$ and any $%
	f\in C^{1}\left( B\right) $,
	\begin{equation}
		\int_{B}|f-f_{B}|^{2}d \mu \lesssim r^{2}\,\int_{B}|\nabla f|^{2}d \mu ,
		\tag{PI}  \label{pi}
	\end{equation}%
	where $f_{B}=\frac{1}{\mu(B)}\int_{B} f d \mu$.
\end{itemize}

In the above and below, for functions $U, V$, we use $U\lesssim V$ (or $U\gtrsim V$) to stand for $U(x)\leq cV(x)$ (resp. $U(x)\geq cV(x)$) for a positive constant $c$ and for each parameter $x$, and $U\asymp V$ to mean that both $U\lesssim V$ and $V\lesssim U$  hold. We also use $c, C, c_1, c_2,\ldots$ to denote generic positive constants whose values are unimportant.

	Denote by $p_t(\cdot, \cdot)$ the heat kernel associated with 
	{$\Delta$} on $M$.
	{The classical works of Grigor'yan \cite{Gri} and Saloff-Coste \cite{SC}} show that the combination of (VD) and (PI) is  equivalent to the
	following Li-Yau Gaussian type heat kernel estimates:
	\[\frac{1}{\mu(B(x,\sqrt{t}))}e^{-\frac{d^2(x,y)}{c_2t}}\lesssim p_t(x,y)\lesssim \frac{1}{\mu(B(x,\sqrt{t}))}e^{-\frac{d^2(x,y)}{c_1t}}.\]
	Via subordination theory, this leads to estimates for the heat kernel and Green kernel of the fractional Laplacian (see Definition \ref{defi-Green-alpha} and Proposition \ref{prop-g-alph}). Consequently,  assuming (VD) and (PI),  $	(-\Delta)^{\alpha}$ is transient if and only if
	there exists some $x_0\in M$ such that (see Proposition \ref{prop-g-alph})
	\[\int_{1}^{+\infty }\frac{t^{2\alpha-1} d  t}{\mu (B(x_0,t))}<\infty. \]

Our first main result is the following sharp criterion for the existence of positive solutions to (\ref{eq-non-local-integral-inequality}) under conditions (VD) and (PI). 
\begin{theorem}\label{thm-main}\rm
	Let $(M,g)$ be a complete connected non-compact Riemannian manifold satisfying conditions (VD) and (PI). Assume that 
	for some $x_0\in M$,
	\[\int_{1}^{+\infty }\frac{t^{2\alpha-1} d  t}{\mu (B(x_0,t))}<\infty. \] 
	Then there exists a positive solution to (\ref{eq-non-local-integral-inequality}) if and only if  
	there exist $o\in M$, $r_{0}>0$ 
	such that the following two
	conditions hold:  
	\begin{equation}
		\int_{r_{0}}^{+\infty }\left[ \int_{r}^{+\infty }\frac{t^{2\alpha-1} d  t}{\mu (B(o,t))}%
		\right] ^{q-1}\frac{\sigma (B(o,r))}{\mu (B(o,r))}r^{2\alpha-1} d  r<\infty ,
		\label{cond-int1}
	\end{equation}%
	and
	\begin{equation}
		\sup_{x\in M, r>r_0}\left[ \int_{0}^{+\infty}\frac{\sigma (B(x,s)\cap B(o, r))}{\mu (B(x,s))}%
		\,s^{2\alpha-1} d  s\right] \,\left[ \int_{r}^{+\infty }\frac{t^{2\alpha-1} d  t}{\mu (B(o,t))}\right]
		^{q-1}<+\infty.
		\label{cond-int2}
	\end{equation}%
\end{theorem}
	
	For the special case that $\sigma=\mu$, the above criterion can be written in a more concise form.
	\begin{theorem}\label{thm-main-mu}\rm
		Let $(M, g)$ be a complete connected non-compact Riemannian manifold satisfying conditions (VD) and (PI). Assume that
		for some $x_0\in M$,
		\[\int_{1}^{+\infty }\frac{t^{2\alpha-1} d  t}{\mu (B(x_0,t))}<\infty. \]
		Then there exists  a lower semi-continuous function $u: M\rightarrow (0,+\infty)$ such that
		\begin{equation}\label{eq-non-local-integral-inequality-mu}
			u(x)\geq \int_{M}G^{(\alpha)}(x,y)u^{q}(y) d \mu (y), 
		\end{equation}		
		for all $x\in M$,	if and only if
		there exist $o\in M$ and $r_{0}>0$ such that
		\begin{equation}
			\int_{r_{0}}^{+\infty }\frac{r^{2\alpha q-1}}{[\mu (B(o,r))]^{q-1}} d  r<\infty.
			\label{cond-int1b}
		\end{equation}
	\end{theorem}
	\begin{remark}\rm
		As we will see from the proof in Section \ref{sect-proofs-vol},  the condition (\ref{cond-int1}) implies (\ref{cond-int2}) in the special case that $\sigma =\mu$. This implication does not hold in general, as can be seen from the special case that $\sigma=\delta_{o}$, the Dirac measure concentrated at $o$. 
	\end{remark}
Another interesting special case concerning (\ref{eq-non-local-differential-inequality}) arsing from Euclidean space is the so called fractional Hardy-H\'{e}non 
equations
\begin{align}\label{eq-Henon}
	(-\Delta)^{\alpha}u=|x|^{\gamma}u^q, \quad\mbox{in $\mathbb{R}^n$},
\end{align}
where $\gamma >-2\alpha, n>2\alpha$. The existence and non-existence of different types of solutions (including stable and unstable solutions, and positive distributional solutions) to (\ref{eq-Henon}) and related problems were studied
 in a vast body of literature, let us refer to the papers \cite{ BQ, DDW17, HIK21, HIK23, LB19, MitPoh} and the references therein.

 Theorem \ref{thm-main} leads to the following corollary, which should be well known in the literature. However, we are not able to locate a precise reference. 
\begin{corollary}\label{thm-Henon}\rm
	Consider the Euclidean space $\mathbb{R}^n$ with $\mu$ being the Lebesgue measure. Fix $\alpha \in (0,1)$ and assume that $n>2\alpha$. Let $\gamma>-2\alpha$ be a fixed parameter. Then there is  a lower semi-continuous function $u: M\rightarrow (0,+\infty)$ which is a super-solution to (\ref{eq-Henon}) such that
	\[u(x)\geq \int_{M}G^{(\alpha)}(x,y)u^{q}(y) \lvert y\rvert^{\gamma}d \mu (y),\] if and only if
	\[q> \frac{n+\gamma}{n-2\alpha}.\] 
\end{corollary}
\begin{remark}\rm
	In the above specific settings, the corresponding Radon measure $\sigma$ is absolutely continuous with respect to $\mu$ with density function 
	$\lvert\cdot\rvert^{\gamma}$. The condition $\gamma>-2\alpha$ guarantees local finiteness of $\sigma$, and here $n>2\alpha$ restriction is simply the transience condition.
\end{remark}

Now we briefly review some closely related classical results.
First for the Laplace-Beltrami operator on a complete manifold, consider the semilinear differential inequality
\begin{equation}
	\Delta u+ u^{{ q}}\sigma\leq0.\label{semi-ineq}
\end{equation}
When $\sigma=\mu$, Grigor'yan and Sun  proved   in  \cite{GS1} that if for some $o\in M$ and all large enough $r$,
\begin{equation}\label{vol-sem}
	\mu(B(o,r))\lesssim  r^{\frac{2q}{q-1}}(\ln r)^{\frac{1}{q-1}},
\end{equation}
then the only non-negative solution to (\ref{semi-ineq}) is zero. Here we note that there is no additional assumption on manifolds in 
their
result.
Later in \cite{GSV}, Grigor'yan, Sun and Verbitsky obtained the sharp integral type criteria for the existence of positive solutions to (\ref{semi-ineq}),  under the additional conditions (VD) and (PI).   Our main result, Theorem \ref{thm-main}, is greatly inspired by \cite{GSV}.
We emphasize that a much higher regularity result than Theorem \ref{thm-main-mu} is obtained in \cite{GSV}, namely,
(\ref{semi-ineq}) has a $C^2$ positive solution if and only if
\[\int_{r_{0}}^{+\infty }\frac{r^{2q-1}}{[\mu (B(o,r))]^{q-1}} d  r<\infty.\]
In the Euclidean setting $M=\mathbb{R}^n$, Wang and Xiao in \cite{XW} obtained that when $n>2\alpha$,   non-negative solutions
	to (\ref{eq-non-local-differential-inequality}) are all trivial if and only if ${ q}\le\frac{n}{n-2\alpha}$.
Other generalizations to  fractional quasilinear differential inequality in Euclidean space can be found in \cite{LSX}.
	
	{In the following we explain the main strategy to prove Theorem \ref{thm-main} and Theorem \ref{thm-main-mu},  and the intermediate results involved. We follow the main ideas of Kalton and Verbitsky in the work \cite{KV}.} 
	
As the fundamental step,
	by tools from potential theory, we show that the existence of a positive solution to (\ref{eq-non-local-integral-inequality}) is equivalent to the existence of a certain ``minimal" positive solution, see Theorem \ref{thm-equiv-minimal}. 
	The crucial condition is the so-called \emph{quasi-metric}  property (see Definition \ref{defi-quasi-metric}; also called (3G) condition in some literature, see \cite{GSV}) for the Green kernel $G^{(\alpha)}$:
	there exists a constant $\kappa\geq1$ such that
	\begin{equation}
		G^{(\alpha)}(x,y)\wedge G^{(\alpha)}(y,z)\leq \kappa G^{(\alpha)}(x,z),\quad \mbox{for all $x,y,z\in M$}.
\end{equation}
	
	Fix some $o\in M$ and $a>0$, set
	\begin{equation*}
		m(x)=m_{a,o}(x)= G^{(\alpha)}(x,o)\wedge a^{-1}.
	\end{equation*}
	\begin{theorem}\label{thm-equiv-minimal}\rm
		Let $M$ be a complete connected non-compact Riemannian manifold. Assume that the fractional Laplacian $	(-\Delta)^{\alpha}$ is transient,
		and suppose that $G^{(\alpha)}$ is quasi-metric. Then the followings are equivalent:
			\begin{enumerate}
				\item There exists a 
				positive solution to (\ref{eq-non-local-integral-inequality}) in the sense of Definition \ref{defi-int-ineq}.
				\item There exists a lower semi-continuous function $u: M\rightarrow (0,+\infty)$ such that (\ref{eq-non-local-integral-inequality}) holds for all $x\in M$.
				\item For $\sigma$-a.e. $x\in M$,	\[m(x)\gtrsim  \int_{M}
				G^{(\alpha)}(x,y) m^q(y) d \sigma(y).
				\]
				\item For all $x\in M$,	\[m(x)\gtrsim  \int_{M}
				G^{(\alpha)}(x,y) m^q(y) d \sigma(y).
				\]
			\end{enumerate}
	\end{theorem}
	
The above theorem allows us to relate the existence of positive solutions to (\ref{eq-non-local-integral-inequality}) with certain estimates on the Green kernel.
\begin{theorem}\label{thm-equiv-integrability}\rm
	Let $M$ be a complete connected non-compact  Riemannian manifold. Assume that the fractional Laplacian $	(-\Delta)^{\alpha}$ is transient, and suppose that $G^{(\alpha)}$ is quasi-metric. Then there exists a  positive solution to (\ref{eq-non-local-integral-inequality}) if and only if the following two conditions hold:
	\begin{equation}\label{mathsection-1}
		\int_{M}m^q(x)d  \sigma (x)<\infty ,
	\end{equation}
	and
	\begin{equation} \label{mathsection-2}
		\sup_{x\in M}\int_{\{y\in M:\,G^{(\alpha)}(o,y)>r^{-1}\}}G^{(\alpha)}(x,y) d  \sigma (y)\lesssim r^{q-1},
	\end{equation}
	for all $r>a$.	
\end{theorem}
Based on the above abstract framework, explicit estimates for the Green kernel come into play.   Under conditions (VD) and (PI), the Green kernel $G^{(\alpha)}$ for $(-\Delta)^{\alpha} $ satisfies the following estimate (see Proposition \ref{prop-g-alph}):
\begin{equation}\label{G-alpha-app-0}
	G^{(\alpha)}(x,y)\asymp
	\int_{d(x,y)}^{+\infty }\frac{t^{2\alpha-1}d t}{\mu (B(x,t))},\quad x,y\in
	M.
\end{equation}
In particular, this implies the quasi-metric property needed, see Corollary \ref{g-alp-qm}. 
	Given (\ref{G-alpha-app-0}), the conditions (\ref{mathsection-1}) and (\ref{mathsection-2}) essentially amount to the corresponding growth type criteria (\ref{cond-int1}) and (\ref{cond-int2}).

Now we turn back to the original differential type inequality (\ref{eq-non-local-differential-inequality}). We are not able to find in the literature  the equivalence for the existence of positive solutions to (\ref{eq-non-local-differential-inequality}) and to the integral type inequality (\ref{eq-non-local-integral-inequality}). Here we propose a certain type of weak solutions as a partial solution to this problem.
\begin{definition}\rm
	Let $M$ be a complete connected non-compact Riemannian manifold. Assume that the fractional Laplacian $	(-\Delta)^{\alpha}$ is transient and that $\sigma$ is \emph{absolutely continuous} with respect to the Riemannian measure $\mu$. Denote by $(\mathcal{E}^{(\alpha)}, \mathcal{F}^{(\alpha)})$ the Dirichlet form associated with $(-\Delta)^{\alpha}$. Let $\mathcal{F}^{(\alpha)}_e$ be the corresponding extended Dirichlet space (see Section \ref{sect-frac-Lap} and \cite{FOT}). 
	A  non-negative function $v\in \mathcal{F}^{(\alpha)}_e$ is said to be a positive solution to  (\ref{eq-non-local-differential-inequality}) in $\mathcal{F}^{(\alpha)}_e$-sense if
$v>0$ $\sigma$-a.e., and
	\[\mathcal{E}^{(\alpha)}(v, \varphi)\ge \int_{M} v^q \varphi d \sigma,\]
	for each $\varphi\in \mathcal{F}^{(\alpha)}\cap C_c(M)$ with $\varphi\ge 0$. 
\end{definition}
\begin{remark}\rm
	Note that as an element of $\mathcal{F}^{(\alpha)}_e$, $v$ is only $\mu$-a.e. well defined. Also note by definition $v$ is $\mu$-a.e. finite. The assumption on absolute continuity of $\sigma$ can be weakened to that $\sigma$ is a smooth measure (\cite[p.83 (S.1)(S.2)]{FOT}).
\end{remark}

Utilizing the Dirichlet form theory, we can show that the desired equivalence for the existence of positive weak solutions to (\ref{eq-non-local-differential-inequality}) with that for (\ref{eq-non-local-integral-inequality}).
\begin{theorem}\label{thm-equivalence}\rm
	Let $M$ be a complete connected non-compact Riemannian manifold. Assume that $\sigma$ is absolutely continuous with respect to $\mu$, with a continuous density function $\theta$. Assume that the fractional Laplacian $	(-\Delta)^{\alpha}$ is transient, and suppose that $G^{(\alpha)}$ is quasi-metric. Then there exists a positive solution to (\ref{eq-non-local-integral-inequality})  if and only if there is a  positive solution to  \eqref{eq-non-local-differential-inequality} in $\mathcal{F}^{(\alpha)}_e$-sense.
\end{theorem}


The regularity is expected to be improved further, ideally to smooth strong solutions. We manage to do this when $\sigma=\mu$, under an additional assumption that the heat semigroup $\{P_t\}_{t\ge 0}$ associated with $\Delta$  is Feller.   This means that $\{P_t\}_{t\ge 0}$ induces a strongly continuous semigroup on the Banach space $C_{\infty}(M)$, the completion of $C_c(M)$ with respect to the $\norm{\cdot}_{\sup}$-norm(see \cite{Levy}). We denote this semigroup on $C_{\infty}(M)$ by $\{P_t\}_{t\ge 0}$ as well, since they are both determined by the heat kernel  $\{p_{t}(x,y)\}_{t> 0}$. Fix some $\alpha\in (0,1)$ and assume that the fractional Laplacian $	(-\Delta)^{\alpha}$ (on $L^2(M, \mu)$) is transient. The subordination construction applied to  $\{P_t\}_{t\ge 0}$ on $C_{\infty}(M)$ (as a Banach space, cf. \cite[Chapter 6]{Sato} and \cite{Bochner}) leads to a strongly continuous semigroup in the same way as the $L^2(M, \mu)$ setting. The subordinated semigroup on $C_{\infty}(M)$ is denoted by $\{P_t^{(\alpha)}\}_{t\ge 0}$, with generator $	-(-\Delta)^{\alpha}$ (on $C_{\infty}(M)$). 
\begin{theorem}\label{thm-strong-solution}
\rm
Let $M$ be a complete connected non-compact Riemannian manifold. 
Assume that the fractional Laplacian $	(-\Delta)^{\alpha}$ is transient and that the Green kernel $G^{(\alpha)}$ is quasi-metric. Assume further that the heat semigroup $\{P_t\}_{t\ge 0}$ associated with $\Delta$  is Feller, and that \[\lim_{d(x,o)\rightarrow +\infty}G^{(\alpha)}(x,o)=0,\] for some fixed reference point $o\in M$. Then there exists a positive solution to (\ref{eq-non-local-integral-inequality}),  if and only if there is a positive smooth function $h$ on $M$ that is in the domain of $(-\Delta)^{\alpha}$ (on $C_{\infty}(M)$) and satisfies
\[(-\Delta)^{\alpha}h\ge h^q.\]
\end{theorem}

\begin{remark}\rm
The additional conditions above are somehow indirect and restrictive. Currently we have to work in this ad hoc way due to lack of a good reference for the regularity of fractional Laplacian on manifolds. In any case, under conditions (VD) and (PI), the Feller property and the asymptotic decay of $G^{(\alpha)}(\cdot,o)$ both hold (see \cite[Theorem 7.1]{Pigola-Setti}). Thus in Theorem \ref{thm-main-mu} the solution can be understood as a smooth one.
\end{remark}

The paper is organized as follows. In Section \ref{sect-frac-Lap}, we introduce the fractional Laplacian on manifolds and show some basic facts about the corresponding Green function. Section \ref{sect-potential} is devoted to a summary of the  potential theoretic tools, the proofs of which are postponed to Appendix \ref{appendix-potential}. 
Theorem \ref{thm-equiv-minimal} is proved in Section \ref{sect-equivalence}. In Sections \ref{sect-equiv-integrability} and \ref{sect-proofs-vol}, we finish the proofs of our main result Theorem \ref{thm-main} and its consequences, Theorem \ref{thm-main-mu} and Theorem \ref{thm-Henon}, based on the preparations from Theorem \ref{thm-equiv-integrability}. The issue of lifting regularity of solutions is addressed in Section \ref{sect-solutions} and Section \ref{sect-strong-solutions}. In Appendix \ref{appendix-DF}, we collect some basic notions and facts about Dirichlet forms, necessary for Section \ref{sect-solutions}.
\section{The fractional Laplacian on manifolds}
\label{sect-frac-Lap}
The fractional Laplacian has attracted an increasing amount of attention in the past decade, see the monograph \cite{Chen-book} of Chen, Li and Ma for a recent survey and further references. In the Euclidean setting, there are various ways to define the fractional Laplacian,  see the original work of Molcanov and Ostrovskii \cite{MO} from the probability theory side, and Caffarelli and Silvestre \cite{CS} from the PDE side. See also  Kwa\'{s}nicki \cite{MK} which explains the construction from various perspectives. 
On non-compact manifolds, the construction of fractional Laplacian has also been studied  through extention method by Banica,  Gonzal\'{e}z, and  S\'{a}ez in \cite{BGS},
and integral representation method by Alonso-Or\'{a}n, C\'{o}rdoba, and Mart\'{i}nez in \cite{ACD}.

For our general setting, it is most convenient to construct the fractional Laplacian via the classical subordination theory of Bochner \cite{Bochner} (see also \cite[Section 32]{Sato} and \cite[Chapter 13]{Bernstein}). 
Note that the subordination approach has also been taken by del Teso,  G\'{o}mez-Castro, and V\'{a}zquez in \cite{TGV} for non-local operators. Here the kernel point of view is stressed on, to be better adapted to the potential theoretic approach. 

Let $(M, g)$ be a complete connected non-compact Riemannian manifold with Riemannian measure $\mu$. Consider the Laplace-Beltrami operator $\Delta=\Div\cdot\nabla$ on $C_c^{\infty}(M)$. It is well known that $\Delta$ is essentially self-adjoint on $L^2(M, \mu)$, and we denote its unique self-adjoint extension by $\Delta$ as well (see \cite[Chapters 3 and 4]{G} for more information). 

Consider the spectral representation of $-\Delta$:
\[-\Delta=\int_{0}^{\infty} \lambda d E_{\lambda},\]
where $\{E_{\lambda}\}_{\lambda\ge 0}$ is the projection operator valued measure associated with $-\Delta$. The heat semigroup generated by $\Delta$,  $P_t =e^{t\Delta}$, can then be represented as
\[P_t=\int_{0}^{\infty} e^{-t\lambda} d E_{\lambda}. \]
The heat semigroup admits positive smooth densities $\{p_{t}(x,y)\}_{t> 0}$, namely
\[P_t f(x)=\int_M p_t(x,y) f(y) d\mu(y), ~~~\forall t>0, x\in M,\]
which is valid for $f\in L^2(M,\mu)$. 
The Green operator corresponding to $-\Delta$ is defined by
\[G=\int_{0}^{\infty}P_t d t,\]
with integral density \[G(x,y)=\int_{0}^{\infty}p_t(x,y) dt,\]
which is possibly infinite  everywhere.

Fix $\alpha\in (0,1)$, we can define the fractional Laplacian $(-\Delta)^{\alpha}$ as
\begin{equation}\label{frac-lap}
	(-\Delta)^{\alpha}=\int_{0}^{\infty} \lambda^{\alpha} d E_{\lambda}.
\end{equation}
By subordination theory, $(-\Delta)^{\alpha}$ is self-adjoint and generates another semigroup, which we denote by $\{P_t^{(\alpha)}\}_{t\ge 0}$. 

	Denote by $\{\eta^{(\alpha)}_{t}(\cdot)\}_{t\ge 0}$ the convolution semigroup associated with the Bernstein function $\lambda\mapsto \lambda^{\alpha}$ via Laplace transform, that is
\[e^{-t \lambda^{\alpha}}=\int_{0}^{\infty} e^{-s\lambda} \eta^{(\alpha)}_{t}(s) ds,\] see \cite{Bernstein, Sato}.
Then the new semigroup $\{P_t^{(\alpha)}\}_{t\ge 0}$ can  be represented in terms of the original one as
\[P_t^{(\alpha)}=	\exp\left(-t(-\Delta)^{\alpha}\right)=\int_{0}^{\infty} \exp\left(-t \lambda^{\alpha}\right) d E_{\lambda}=\int_{0}^{\infty} \eta^{(\alpha)}_{t}(s)  P_s ds.\]
The following is clear by Laplace transform (cf. \cite[page 97]{Bogdan-etc}):
	\[\int_{0}^{\infty}\eta^{(\alpha)}_{t}(s) dt=\frac{1}{\Gamma(\alpha
		)}s^{\alpha -1}.\]
%

	Note that the Green operator $G^{(\alpha)}$ formally satisfies
\[G^{(\alpha)}=\int_{0}^{\infty} P_t^{(\alpha)} dt=	\int_{0}^{\infty}\int_{0}^{\infty} \eta^{(\alpha)}_{t}(s)  P_s ds dt=\frac{1}{\Gamma(\alpha
	)}\int_{0}^{\infty} s^{\alpha -1} P_s ds.\]
This motivates the following definition.
\begin{definition}\label{defi-Green-alpha}\rm
	The Green kernel $G^{(\alpha)}(\cdot, \cdot)$ associated with the fractional Laplacian $	(-\Delta)^{\alpha}$ is defined as
	\begin{equation}\label{eq-Green-density}
		G^{(\alpha)}(x,y)=\frac{1}{\Gamma(\alpha
			)}\int_{0}^{\infty} s^{\alpha -1} p_s(x,y) ds,~~~~x,y\in M.
	\end{equation}
	We say that the fractional Laplacian $	(-\Delta)^{\alpha}$ is transient if there exists a non-negative measurable function $f$ on $M$ with $\mu\left(\{x\in M: f(x)>0\}\right)>0$ such that
	\begin{equation}\label{eq-local-int}
		\int_{M} G^{(\alpha)}(x, y) f(y) d\mu(y)<+\infty,
	\end{equation}
	for  $\mu$-a.e. 
	$x\in M$.
\end{definition}
\begin{remark}\rm
	The definition of transience comes from probabilistic considerations (see \cite[Definition 2.1.1]{Chen-Fuku} and  \cite[Lemma 1.6.4]{FOT}). Note that the heat kernel is  everywhere positive, and hence the corresponding semigroup (or equivalently, the Markov process) is irreducible.	
\end{remark}

Now we turn to concrete properties of the Green kernel.
\begin{proposition}\label{prop-lsc}\rm
	Let $(M, g)$ be a complete connected non-compact Riemannian manifold. 
	Then $G^{(\alpha)}(\cdot,\cdot): X\times X\rightarrow [0,\infty]$ is lower semi-continuous.
\end{proposition}
\begin{proof}
	Note that
	\begin{align*}
		G^{(\alpha)}(x,y) =&\frac{1}{\Gamma(\alpha
			)}\int_{0}^{\infty} s^{\alpha -1} p_s(x,y) ds \\
		=&\lim_{N\rightarrow +\infty}\frac{1}{\Gamma(\alpha
			)}\int_{1/N}^{N} s^{\alpha -1} p_s(x,y) ds
	\end{align*}
	is the monotone increasing limit of a sequence of functions continuous in $(x,y)$. It follows that $G^{(\alpha)}(\cdot,\cdot)$ is lower semi-continuous.	
\end{proof}
As mentioned in the Introduction, for a complete connected non-compact Riemannian manifold $(M, g)$,  the combination of (VD) and (PI) is  equivalent to the
following Li-Yau Gaussian type heat kernel estimates:
\[\frac{1}{\mu(B(x,\sqrt{t}))}e^{-\frac{d^2(x,y)}{c_2t}}\lesssim p_t(x,y)\lesssim \frac{1}{\mu(B(x,\sqrt{t}))}e^{-\frac{d^2(x,y)}{c_1t}}.\]
This naturally leads to estimates on the Green kernel.
\begin{proposition}\label{prop-g-alph}\rm
	Let $(M, g)$ be a complete connected non-compact Riemannian manifold satisfying conditions (VD) and (PI). Then the Green kernel for $(-\Delta)^{\alpha} $ satisfies
	\begin{equation}\label{G-alpha-app}
		G^{(\alpha)}(x,y)\asymp
		\int_{d(x,y)}^{+\infty }\frac{t^{2\alpha-1}d t}{\mu (B(x,t))},\quad x,y\in
		M.
	\end{equation}
	Moreover,  the fractional Laplacian $(-\Delta)^{\alpha} $ is transient if and only if for some $x_0\in M$,	
	\begin{equation}\label{eq-int-finite}
		\int_{1}^{+\infty }\frac{t^{2\alpha-1}d t}{\mu (B(x_0,t))}<\infty. 
	\end{equation}
\end{proposition}
\begin{proof}
	The proof of (\ref{G-alpha-app}) is essentially the argument of \cite[Lemma 5.50]{G06}  applied to
	\[F(t)=t^{2-2\alpha}\mu(B(x,t)),\]
	whence $F(t)$ satisfies the doubling condition.
	
	According to Definition \ref{defi-Green-alpha}, if the fractional Laplacian $(-\Delta)^{\alpha} $ is transient,  it is clear 
		that $G^{(\alpha)}(x_0, y_0)<+\infty$ for some pair of $x_0\ne y_0 \in M$. By  (\ref{G-alpha-app}), this implies that
		\[\int_{d(x_0, y_0)}^{+\infty }\frac{t^{2\alpha-1}d t}{\mu (B(x_0,t))}<\infty.\]Then (\ref{eq-int-finite}) follows by a simple change of variables.
	
	Suppose that (\ref{eq-int-finite}) holds. 
		For each $x\in M$ and $t\ge 1$, noting that $c(x)\mu(B(x,t))\ge  \mu(B(x, t+d(x,x_0)))\ge \mu(B(x_0, t))$ for some $c(x)>0$, we then obtain for all $x\in M$,
		\begin{align*}
			\int_{1}^{+\infty }\frac{t^{2\alpha-1}d t}{\mu (B(x,t))} \le c(x) \int_{1}^{+\infty }\frac{t^{2\alpha-1}d t}{\mu (B(x_0,t))}<+\infty.
		\end{align*}
		Let $f=1_{B(x_0, 1)}$. For each $x\in M$, we have 
		\begin{align*}
			\int_{M} G^{(\alpha)}(x, y) f(y) d\mu(y)\lesssim & \int_{B(x_0, 1) }	\int_{d(x,y)}^{+\infty }\frac{t^{2\alpha-1}}{\mu (B(x,t))}d t d\mu(y)\\
			= & 	\int_{0}^{+\infty }\frac{\mu(B(x, t)\cap B(x_0, 1))}{\mu (B(x,t))} t^{2\alpha-1}d t\\
			\le &\int_{0}^{1} t^{2\alpha-1}d t+\mu(B(x_0,1))\cdot\int_{1}^{+\infty }\frac{t^{2\alpha-1}}{\mu (B(x,t))}d t\\
			<&+\infty.
		\end{align*}
		Hence the fractional Laplacian $(-\Delta)^{\alpha} $ is transient.
\end{proof}

An important consequence of (\ref{G-alpha-app}) is the quasi-metric property for $G^{(\alpha)}$.
\begin{corollary}\label{g-alp-qm}\rm
	Under the assumptions of Proposition \ref{prop-g-alph},
	there exists a constant $\kappa\geq1$ such that
	\begin{equation}\label{cond-qm-G-alpha}
		G^{(\alpha)}(x,y)\wedge G^{(\alpha)}(y,z)\leq \kappa G^{(\alpha)}(x,z),\quad \mbox{for all $x,y,z\in X$}. \tag{qm}
	\end{equation}
	
\end{corollary}

\begin{proof}
	Without loss of generality, we can assume that $d(x,y)\ge d(y,z)$, whence $d(x,z)\le 2d(x,y)$.
	
		It follows that
		\begin{align*}
			&\min\left\{\int_{d(x,y)}^{+\infty }\frac{t^{2\alpha-1}d t}{\mu (B(y,t))}, \int_{d(y,z)}^{+\infty }\frac{t^{2\alpha-1}d t}{\mu (B(y,t))}\right\}\\	=&\int_{d(x,y)}^{+\infty }\frac{t^{2\alpha-1}d t}{\mu (B(y,t))}\\
			\le &C\int_{d(x,y)}^{+\infty }\frac{t^{2\alpha-1}d t}{\mu (B(x,t))}\\
			\le &C'\int_{2d(x,y)}^{+\infty }\frac{t^{2\alpha-1}d t}{\mu (B(x,t))}\\
			\le &C'\int_{d(x,z)}^{+\infty }\frac{t^{2\alpha-1}d t}{\mu (B(x,t))}.
		\end{align*}
		By (\ref{G-alpha-app}), this  implies (\ref{cond-qm-G-alpha}).
	
\end{proof}	


\section{Potential theoretic tools}\label{sect-potential}
In this section, we collect some necessary tools from potential theory as developed in \cite{GV1, GV2, KV, QV, VW}. Some modifications on the arguments are needed for our slightly different settings. In particular,  we completely avoid the usage of balayage theory. For the sake of completeness, we include these details in Appendix \ref{appendix-potential}.

Let $(X, d)$ be a locally compact, separable metric space. 
Note that $X$ is $\sigma$-compact (see  \cite[3.18.3]{Dieudonne}). 
Denote by $\mathcal{M}^{+}(X)$ the space of Radon measures on $X$.  By the Riesz representation theorem,  $\mathcal{M}^{+}(X)$ coincides with the class of locally finite Borel measures on $X$.  Hence each $\nu\in \mathcal{M}^{+}(X)$ is $\sigma$-finite.

For $\nu\in \mathcal{M}^{+}(X)$ and a locally bounded, non-negative  measurable function $f$, we denote by  $f\nu$ the indefinite integral of $f$ against $\nu$. It is clear that $f\nu\in  \mathcal{M}^{+}(X)$ as well. Let $A\subset X$ be a  measurable set. We say that $\nu\in  \mathcal{M}^{+}(X)$ is concentrated on $A$ if for any non-negative measurable function $g$ on $X$, $\int_X g d\nu=\int_A g d\nu$. For example, in our current setting, each $\nu\in  \mathcal{M}^{+}(X)$ is concentrated on $\text{supp~} \nu$, and  $f\nu$ is concentrated on $\{f>0\}$.

By a kernel on $X$ we mean a  measurable function $K(\cdot, \cdot): X\times X \rightarrow [0,+\infty]$.
For $\nu\in \mathcal{M}^{+}(X)$ and a non-negative measurable function $f$, we denote
\[K_{\nu}f(\cdot) =\int_X K(\cdot, y)f(y)d\nu(y).\] 
{For instance, the integral inequality (\ref{eq-non-local-integral-inequality}) for the Green function $G^{(\alpha)}$ can be written in a more concise way as:
	\[u\ge G^{(\alpha)}_{\sigma}(u^q).\]}
Throughout the article, we set $0\cdot (+\infty)=0$. Note that under this convention \[\lim_{k\rightarrow \infty} a_k \cdot b_k=\lim_{k\rightarrow \infty} a_k \cdot \lim_{k\rightarrow \infty} b_k\] for monotone increasing sequences $\{a_k\}_{k\ge 1}, \{b_k\}_{k\ge 1} $ in $[0,+\infty]$. 

We will make use of the following properties that a kernel $K$  possibly satisfies:
\begin{enumerate}
	\item $K$ is  positive if $K(x,y)\in (0,+\infty]$ for each $x,y\in X$;
	\item $K$ is symmetric if $K(x,y)=K(y,x)$ for each $x,y\in X$;
	\item $K$ is lower semi-continuous if $K$ is so as a function $X\times X\rightarrow[0,+\infty]$;
	\item $K(x,\cdot)$ is $\sigma$-finite with respect to $\nu \in \mathcal{M}^{+}(X)$ if $K(x,\cdot)\nu$ is a $\sigma$-finite measure.
	\item $K(x,\cdot)$ is locally finite with respect to $\nu \in \mathcal{M}^{+}(X)$  if $K(x,\cdot)\in L^{1}_{\text{loc}}(X, \nu)$. Note that in our setting this implies that $K(x,\cdot)$ is $\sigma$-finite  with respect to $\nu$.
		\item $K$ is $\sigma$-finite (locally finite) with respect to $\nu$ if $K(x,\cdot)$ is $\sigma$-finite (locally finite) with respect to $\nu$ for each $x\in X$. 
\end{enumerate}

The class of kernels that satisfy a weak version of maximum principle plays a central role. A series of estimates has been developed for them by
Grigor'yan and Verbitsky, see \cite{GV1, GV2} for details.
\begin{definition}[Weak Maximum Principle]\label{weak-max}\rm
	A kernel $K$ on a locally compact separable metric space $(X, d)$ 
	is said to satisfy the weak maximum principle with constant $b\ge 1$, if for each $\nu \in \mathcal{M}^{+}(X)$ and each Borel set $A\subset X$ such that $\nu$ is concentrated on $A$, we have
	\begin{eqnarray}\label{wm-1}
		K_{\nu}1\leq1\quad\mbox{in $A$}\quad\Rightarrow\quad K_{\nu}1\leq b\quad\mbox{in $X$}.
	\end{eqnarray}
\end{definition}
This property is stable under taking  limit of an increasing sequence of kernels.
\begin{lemma}\label{lem-kernel-limit}\rm
	Let $\{K_n\}_{n\ge 1}$ be an increasing sequence of kernels on a locally compact separable metric space. Suppose that each of them satisfies the weak maximum principle with a common constant $b\ge 1$. Set
	\[K=\lim_{n\rightarrow +\infty} K_n.\] Then $K$ satisfies the weak maximum principle with constant $b\ge 1$.
\end{lemma}

The following  hereditary property of the weak maximum principle is clear from the definition.
\begin{lemma}\label{lem-kernel-multiple}\rm
	Let $K$ be a kernel on $(X, d)$ satisfying the weak maximum principle with constant $b\ge 1$. Suppose $\eta$ is a 
	non-negative measurable function. Define $\tilde{K}$ by $\tilde{K}(x,y)=K(x,y)\eta(y)$. Then $\tilde{K}$ satisfies the weak maximum principle with constant $b$ as well.
\end{lemma}

Sometimes we use the following convenient version of weak maximum principle.
\begin{corollary}\rm\label{cor-max-f}
	Let $K$ be a kernel on $(X, d)$ satisfying the weak maximum principle with constant $b\ge 1$. Suppose $f$ is a
	non-negative measurable function. Let $\nu\in \mathcal{M}^{+}(X)$  be concentrated on a Borel set $A\subset X$.
	Then
	\begin{eqnarray}\label{wm-f}
		K_{\nu}f\leq1\quad\mbox{in $\{f>0\}\cap A$}\quad\Rightarrow\quad K_{\nu}f\leq b\quad\mbox{in $X$}.
	\end{eqnarray}
\end{corollary}

The following rearrangement inequality is the starting point of a series of estimates.
\begin{lemma}\cite[Lemma 2.1]{GV2}\label{rearrange} \rm 
	Let  $(\Omega, \mathcal{A},\omega)$ 
	be a $\sigma\text{-}$finite measure space, and let $0<a=\omega(\Omega)\leq
	+\infty$. Let $f:\Omega\to[0,+\infty]$ be a measurable function. Let $\varphi:[0,a)\to[0,+\infty)$ be
	a continuous, monotone non-decreasing function, and set $\varphi(a):=\lim_{t\to a^{-}}\varphi(t)\in(0,\infty]$.
	Then the following inequality holds:
	\begin{eqnarray}\label{re-arrg}
		\int_0^{\omega(\Omega)}\varphi(t)dt\leq \int_{\Omega}\varphi\left(\omega(\{z\in\Omega:f(z)\leq f(y)\})\right)d\omega(y).
	\end{eqnarray}
\end{lemma}

The following fundamental estimate follows by applying Lemma \ref{rearrange} to a kernel satisfying the weak maximum principle.

\begin{lemma}\label{int-type-ineq}\cite[Lemma 2.5]{GV2} \rm
	Let $(X,d)$ be a locally compact separable metric space.  Let $\nu\in  \mathcal{M}^{+}(X)$.
	Assume that
	$K$ satisfies the weak maximum principle with constant $b\ge 1$.  
	Let $\varphi:[0,+\infty)\to[0,+\infty)$ be
	a continuous,  non-decreasing function, and set $\varphi(+\infty):=\lim_{t\to +\infty}\varphi(t)\in[0,+\infty]$. Then
	\begin{eqnarray}
		\int_0^{K_{\nu}1 (x)}\varphi(t)dt\leq K_{\nu}[\varphi(b K_{\nu}1)](x),
	\end{eqnarray}
 for each $x\in X$ such that $K(x,\cdot)$ is $\sigma$-finite with respect to $\nu$.
\end{lemma}

\begin{corollary}\cite[(2.10)]{GV2}\label{cor-Hardy-p} \rm
	Let $p> 1$. 
	Under the same conditions as in Lemma \ref{int-type-ineq}, the following inequality holds for each $x\in X$ such that $K(x,\cdot)$ is $\sigma$-finite with respect to $\nu$,
	\[(K_{\nu}1(x))^p\le pb^{p-1} (K_{\nu}((K_{\nu}1)^{p-1}))(x).\]
\end{corollary}
We can extract from the proof of Lemma \ref{int-type-ineq} the following fact, which can simplify our presentation.
\begin{lemma}\rm\label{lem-handy}
	Let $(X,d)$ be a locally compact separable metric space.  Let $\nu\in  \mathcal{M}^{+}(X)$.
	Assume that
	$K$ satisfies the weak maximum principle with constant $b\ge 1$.  Let $f$ be a non-negative measurable function on $X$. For each $y\in X$, set
	\[	E_y=\left\{z\in X: K_{\nu}f(z)\leq K_{\nu}f(y)\right\}.\]
	Then for all $x,y\in X$,
	\[K_{\nu}(1_{E_y}f)(x)\le bK_{\nu}f(y).\]
\end{lemma}

The following iterative version of Lemma \ref{int-type-ineq} is obtained in \cite{GV2}.

\begin{lemma}\cite[Lemma 2.7]{GV2}\label{lem-iter}  \rm
	Under the setting of Lemma \ref{int-type-ineq}, define a sequence $\{f_k\}_{k\ge 0}$ of functions on $X$
	by
	\begin{eqnarray}\label{it-f}
		f_0=K_{\nu}1,\quad f_{k+1}=K_{\nu}(\varphi(f_k))\quad\mbox{for  $k\in\mathbb{N}$}.
	\end{eqnarray}
	Set for $t\ge 0$
	\begin{eqnarray}\label{def-psi}
		\psi(t)=\varphi(b^{-1}t)
	\end{eqnarray}
	and define also the sequence $\{\psi_k\}_{k=0}^{\infty}$ of functions on $[0,\infty)$ by $\psi_0(t)=t$ and
	\begin{eqnarray}\label{psi-k}
		\psi_{k+1}(t)=\int_0^t\psi\circ\psi_k(s)ds,\quad\mbox{for $k\in \mathbb{N}$}.
	\end{eqnarray}
	Then, for each
	$k\in \mathbb{N}$
	\begin{eqnarray}\label{est-it}
		\psi_k(f_0(x))\leq f_k(x),
	\end{eqnarray}
for each $x\in X$ such that $K(x,\cdot)$ is $\sigma$-finite with respect to $\nu$.
\end{lemma}
\begin{corollary}\label{cor-psi=tq}\cite[Corollary 2.8]{GV2} \rm
	Under the settings of Lemma \ref{int-type-ineq} and Lemma \ref{lem-iter}, choose $\varphi(t)=t^q$ for some $q>0$. We have  for each $x\in X$ such that $K(x,\cdot)$ is $\sigma$-finite with respect to $\nu$,
	\begin{eqnarray}\label{it-q-1}
		[K_{\nu}1(x)]^{1+q+q^2+\cdots+q^k}\leq b^{q+q^2+\cdots+q^k}c(q,k) f_k(x),
	\end{eqnarray}
	where
	$c(q,k)=\prod_{j=1}^{k}(1+q+q^{2}+\cdots +q^{j})^{q^{k-j}}.$
\end{corollary}
Based on Lemma \ref{lem-iter} and Corollary \ref{cor-Hardy-p}, and following the ideas in \cite[Theorem 3.3]{GV2}, we can obtain the following important estimate.
\begin{theorem}\cite[Theorem 3.2]{GV2}\label{thm-iteration} \rm
	Let $(X,d)$ be a locally compact separable metric space.  Let $\nu\in  \mathcal{M}^{+}(X)$.
	Assume that
	$K$ satisfies the weak maximum principle with constant $b\ge 1$.
Let $g:[1,+\infty]\rightarrow[1,+\infty]$ be a continuous, non-decreasing function. Set for $t\ge 1$,
	\[F(t)=\int_{1}^{t}\frac{d s}{g(s)}.\]
	Suppose that $u:X\rightarrow[1,+\infty]$ is a  measurable function. Assume that there is a measurable set $A\subseteq X$ with $\nu(A^c)=0$ such that for each $x\in A$, $u(x)<+\infty$, and 
		\[u(x)\ge K_{\nu}(g(u))(x)+1\]
		holds.
		Then at each $x\in A$, we have
		\[K_{\nu}1(x)<b\int_{1}^{+\infty}\frac{d t}{g(t)}.\]
\end{theorem}

\begin{corollary}\cite[Corollary 3.4]{GV2}\label{coro-iteration-1} \rm
	Under the setting of Theorem  \ref{thm-iteration}, choose $g(t)=t^q$ for some $q>1$. We have at  each $x\in A$,
	\[K_{\nu}1(x)<\frac{b}{q-1}.\]
\end{corollary}

Theorem \ref{thm-iteration} can be extended to more general types of integral inequalities.
\begin{theorem}[{A variant of \cite[Theorem 5.1]{GV2}}]\label{thm-iteration-h} \rm
	Let $(X,d)$ be a locally compact separable metric space.  Let $\nu\in  \mathcal{M}^{+}(X)$. Let $h$ be a bounded, positive measurable function  on $X$. Assume that the new kernel $\tilde{K}$ defined by $\tilde{K}(x,y)=\frac{K(x,y)}{h(x)h(y)}$ satisfies the weak maximum principle with constant $b\ge 1$.
	
	Let $q>1$ be a constant. Suppose that $u:X\rightarrow (0,+\infty]$ is a 
	measurable function. Assume that there is a measurable set $A\subseteq X$ with $\nu(A^c)=0$ such that for each $x\in A$, $u(x)<+\infty$, and 
		\[u(x)\ge K_{\nu}(u^q)(x)+h(x)\]holds.
		Then at each $x\in A$, we have
		\[K_{\nu}(h^q)(x)<\frac{b}{q-1} h(x).\]
\end{theorem}
\begin{remark}\rm
	The condition that $\tilde{K}$ satisfies the weak maximum principle is closely related to the so-called weak domination principle as in \cite{GV2}.
\end{remark}

Now we consider the quasi-metric property (\ref{g-alp-qm}) for general kernels. The  quasi-metric kernels will be shown to satisfy the weak maximum principle. In particular, the previous developed tools can then be applied to the Green kernel $G^{(\alpha)}$.

\begin{definition}\label{defi-quasi-metric}\rm
	We say a kernel $K$ on $X$ is quasi-metric, if $K$ is symmetric and there exists a constant $\kappa\geq1$ such that
	\begin{equation}\label{cond-qm}
		K(x,y)\wedge K(y,z)\leq \kappa K(x,z),\quad \mbox{for all $x,y,z\in X$}. \tag{qm}
	\end{equation}
\end{definition}
\begin{remark}\rm
	Consider $\rho(x,y)=\frac{1}{K(x,y
		)}$ for $x,y \in X$. Then  (\ref{cond-qm}) asserts that $\rho$ is a  possibly degenerate quasi-metric on $X$.
	It is clear that (\ref{cond-qm}) is equivalent to the so called (3G) inequality in the literature, up to an absolute constant multiple. The current form is  the most convenient for us.
\end{remark}

From (\ref{cond-qm}), an analogue of the classical Ptolemy type inequality holds.
\begin{lemma}\cite[Lemma 2.2]{FNV}\label{lem-Ptolemy} \rm
	Let $K$ be a quasi-metric kernel on $X$ with constant $\kappa\ge 1$. Then
	\begin{equation}\label{eq-Ptolemy}
		(K(x,y)K(o,z))\wedge(K(y,z)K(o,x))\le \kappa^2 (K(x,z)K(o,y)).
	\end{equation}
\end{lemma}
The  weak maximum principle and a variant of it then follow from the quasi-metric property.
\begin{lemma}\label{lem-quasi-weak}\cite[Lemma 3.5]{QV} \rm
	If $K$ is a quasi-metric kernel on $(X, d)$ with constant $\kappa\ge 1$, then $K$ satisfies  the weak maximum principle with constant $b=\kappa$.
\end{lemma}

\begin{lemma}\cite[Lemma 3.4]{QV}\label{lem-quasi-m} \rm
	Let $K$ be a positive, quasi-metric kernel with constant $\kappa\ge 1$. Fix $o\in X$ and $c>0$. Let $k(\cdot)=K(o,\cdot)\wedge c$. Consider $\tilde{K}: X\times X\rightarrow (0,+\infty]$ defined by
	\[\tilde{K}(x,y)=\frac{K(x,y
		)}{k(x)k(y)}, ~~ x,y\in X.\]
	Then for all $x,y,z\in X$,
	\begin{equation}\label{eq-tilde-K}
		\tilde{K}(x,y)\wedge \tilde{K}(y,z)\leq \kappa^2 \tilde{K}(x,z).
	\end{equation}
\end{lemma}

\section{Proof of Theorem \ref{thm-equiv-minimal}}\label{sect-equivalence}
A simple but crucial observation is that quasi-metric kernels satisfy certain minimality. Without appealing to the local Harnack inequality  as in \cite{GSV} or other extra assumptions,
we give a new proof of the following lemma. 

\begin{lemma}\label{lem-min}\rm
	Let $K$ be a lower semi-continuous quasi-metric positive kernel on $(X, d)$ with constant $\kappa\ge 1$. Fix $o\in X$ and $a>0$. Set $m(x)=K(o,x)\wedge a^{-1}$.
	Then for each positive measure $\omega$ of the form $\omega=f\nu$ with $f$ non-negative 
		 measurable and $\nu$ a Radon measure, we have for all $x\in X$,
	\[K_{\omega}1(x)\gtrsim m(x).\]
\end{lemma}
\begin{proof}
	{	First we treat the case that $\omega$ is a Radon measure.}
	Without loss of generality, suppose that  $\omega$ is of compact support, and let $A=\text{supp}(\omega)$.
	Since $K(\cdot,o)$ is lower semi-continuous and positive, it achieves $\min_{y\in A}K(y,o)=:\lambda>0$ on $A$.
	
	When $K(x,o)>\lambda$, we have
	\begin{eqnarray*}
		\kappa K(x,y)\geq K(x,o)\wedge K(y,o)\geq\lambda,\quad\mbox{for $y\in A$}.
	\end{eqnarray*}
	It follows that
	\begin{eqnarray}\label{har-1}
		K_{\omega}1(x)=\int_{A}K(x,y)d\omega(y)\geq \frac{\lambda}{\kappa}\omega(A).
	\end{eqnarray}
	
	While if $K(x,o)\leq \lambda$, we have
	\begin{eqnarray*}
		\kappa K(x,y)\geq K(x,o)\wedge K(y,o)=K(x,o),
	\end{eqnarray*}
	which yields
	\begin{eqnarray}\label{har-2}
		K_{\omega}1(x)=\int_{A}K(x,y)d\omega(y)\geq \frac{\omega(A)}{\kappa} K(x,o).
	\end{eqnarray}
	Combining with (\ref{har-1}) and (\ref{har-2}), we derive
	\begin{eqnarray*}
		K_{\omega}1(x)\geq \frac{\omega(A)}{\kappa}(\lambda \wedge K(x,o))\ge \frac{\omega(A)}{\kappa} \min\{\lambda a, 1\}m(x),
	\end{eqnarray*}
	which completes the proof in this case.
	
	 Now consider $\omega=f\nu$ as assumed. Since $X$ is $\sigma$-compact, there is a compact subset $A$ such that $\omega(A)>0$. Consider $\tilde{\omega}=\mathbf{1}_A(f\wedge 1)\nu$. It is clear that $\tilde{\omega}$ is a locally finite positive Borel measure, and hence a positive Radon measure. The assertion follows from the previous case.
\end{proof}
\begin{remark}\rm
	To the best of our knowledge, it is not clear whether the local Harnack inequality for the Laplace-Beltrami operator still holds for the fractional Laplacian.
\end{remark}

For Theorem \ref{thm-equiv-minimal}, the Green kernel $G^{(\alpha)}$ is assumed to be quasi-metric, so the potential theoretic tools from Section \ref{sect-potential} and  Lemma \ref{lem-min} apply.

Recall that we denote
\[G^{(\alpha)}_{\sigma}(f):=\int_M G^{(\alpha)}(\cdot,y) f(y) d \sigma(y).\]
For simplicity, we write $G^{(\alpha)}(f)$ for $G^{(\alpha)}_{\mu}(f)$.

\begin{proof}[Proof of Theorem \ref{thm-equiv-minimal}]
	We finish the proof by showing that ``(3)$\implies$(4)'', ``(4)$\implies$(2)'',  ``(2)$\implies$(1)'', and  ``(1)$\implies$(3)''.
	
	``(3)$\implies$(4)'': Let $A$ be the set of $x\in X$ such that \[m(x)\gtrsim  \int_{M}
	G^{(\alpha)}(x,y) m^q(y) d \sigma(y).
	\]
	By assumption that $\sigma$ is concentrated on $A$.  Since $G^{(\alpha)}(\cdot, \cdot)$ is quasi-metric, by Lemma \ref{lem-quasi-m} and Lemma \ref{lem-kernel-multiple}, the kernel $\hat{G}^{(\alpha)}(\cdot, \cdot)$ defined by
	\[\hat{G}^{(\alpha)}(x, y)=\frac{1}{m(x)}G^{(\alpha)}(x, y)m(y)^q\]
	satisfies the weak maximum principle. It follows that
	\[m(x)\gtrsim  \int_{M}
	G^{(\alpha)}(x,y) m^q(y) d \sigma(y),
	\] for each $x\in X$, with a possibly different constant.
	
	``(4)$\implies$(2)'': This is clear since $m(x)=G^{(\alpha)}(x,o)\wedge a^{-1}$ is lower semi-continuous,  and is positive and finite everywhere. The inequality (\ref{eq-non-local-integral-inequality}) holds for $cm$ with some suitable constant $c>0$.
	
	``(2)$\implies$(1)'': This is trivial.
	
	``(1)$\implies$(3)'': 	Suppose that $u$ is a positive solution to (\ref{eq-non-local-integral-inequality}) in the sense of Definition \ref{defi-int-ineq}. Then $u: M\rightarrow [0,+\infty]$ is  lower semi-continuous and we can find some measurable set $A\subseteq M$ with $\sigma(A^c)=0$ such that $u(x)\in (0,+\infty)$ and
	\begin{equation}
		u(x)\geq\int_{M}G^{(\alpha)}(x,y)u^{q}(y) d \sigma (y), 
	\end{equation}	
	for each $x\in A$. 
	Note that for each $x\in A$, \[+\infty>u(x)
	\ge\int_{M}G^{(\alpha)}(x,y)u^{q}(y) d \sigma (y)=\int_{A}G^{(\alpha)}(x,y)u^{q}(y) d \sigma (y).\]  Hence $G^{(\alpha)}(x,\cdot)$ is $\sigma$-finite with respect to $\nu$ for each $x\in A$.
	
	Fix some $\varepsilon\in (0,1)$, let $v=\varepsilon u$. Then
	\begin{equation}\label{blacktriangle}
		v\ge G^{(\alpha)}_{\sigma} (v^q)+(\varepsilon^{1-q}-1)G^{(\alpha)}_{\sigma} (v^q)
	\end{equation}
	holds on $A$.
	Since $v$ is lower semi-continuous and $\sigma$-a.e. positive,  by Lemma \ref{lem-min}, there is a constant $C>0$ such that
	\[(\varepsilon^{1-q}-1)G^{(\alpha)}_{\sigma} (v^q)=(\varepsilon^{1-q}-1)G^{(\alpha)}_{v^q\sigma} 1\ge C(\varepsilon^{1-q}-1) m.\]
	Hence $v$ satisfies the following inequality for each $x\in A$,
	\[v(x)\ge G^{(\alpha)}_{\sigma} (v^q)(x)+C m(x).\]
	By Lemma \ref{lem-quasi-m} and Lemma \ref{lem-quasi-weak}, the kernel $\tilde{G}^{(\alpha)}$ defined by
	\[\tilde{G}^{(\alpha)}(x,y)=\frac{G^{(\alpha)}(x,y)}{m(x)m(y)}, ~~x,y\in M\]
	satisfies the weak maximum principle with constant $b=\kappa^2$.
	
	
	By Theorem \ref{thm-iteration-h}, we then obtain the following estimate of $m$ for each $x\in A$
	\[G^{(\alpha)}_{\sigma} ((Cm)^q)(x)\le \frac{b}{q-1} C m(x). \]
	It simplifies to
	\[G^{(\alpha)}_{\sigma} (m^q)\lesssim  m\] on $A$,
	and the assertion follows.
\end{proof}

\section{Proof of Theorem \ref{thm-equiv-integrability}}\label{sect-equiv-integrability}
By Theorem \ref{thm-equiv-minimal}, it suffices to show that the combination of conditions (\ref{mathsection-1}) and (\ref{mathsection-2}) is equivalent to the following inequality
{\[m \gtrsim  G^{(\alpha)}_{\sigma}(m^q).\]}
\begin{proof}[Proof of the ``$\implies$'' direction.]
	As shown in the ``only if'' part of the proof for Theorem \ref{thm-equiv-minimal}, $G^{(\alpha)}$ is locally integrable with respect to $\sigma$.	
	
	Let $s=\frac{q}{q-1}$. Fix an arbitrary non-negative, bounded $f\not\equiv 0$ with compact support.
	Since  $G^{(\alpha)}$ is quasi-metric and consequently satisfies the weak maximum principle with constant $b=\kappa$, we can  apply Corollary \ref{cor-Hardy-p} to obtain
	\begin{align}
		\int_{M} (G^{(\alpha)}_{\sigma} (f))^s m^q  d  \sigma &\le sb^{s-1}\int_{M} G^{(\alpha)}_{\sigma} (f(G^{(\alpha)}_{\sigma} (f))^{s-1}) m^q  d  \sigma\notag\\&=sb^{s-1}\int_{M} f(G^{(\alpha)}_{\sigma} (f))^{s-1} G^{(\alpha)}_{\sigma}( m^q)  d  \sigma\notag\\ &\lesssim sb^{s-1}\int_{M} f(G^{(\alpha)}_{\sigma} (f))^{s-1} m  d  \sigma \label{eq5.1}\\
		&\lesssim sb^{s-1} \norm{f}_{L^s(\sigma)} \left(\int_{M} (G^{(\alpha)}_{\sigma} (f))^s m^q  d  \sigma\right)^{1/q}.\label{eq5.2}
	\end{align}
	Note that $f$ is bounded and compactly supported, while $m$ is lower semi-continuous and positive. For some constant $C_1>0$, we have  $f\le C_1 m^q$. Consequently
	\[G^{(\alpha)}_{\sigma} (f)\le C_1 G^{(\alpha)}_{\sigma} (m^q)\lesssim  C_1 m.\]
	Then we obtain
	\[\int_{M} f(G^{(\alpha)}_{\sigma} (f))^{s-1} m  d  \sigma \lesssim \int_M f C_1^{s-1} m^s d\sigma \lesssim C_1^{s-1} \frac{1}{a^s} \int_M f d\sigma<+\infty.\]
	It follows by \eqref{eq5.1} that
	\[\int_{M} (G^{(\alpha)}_{\sigma} (f))^s m^q  d  \sigma<+\infty,\]
	and hence  by \eqref{eq5.2}
	\[\int_{M} (G^{(\alpha)}_{\sigma} (f))^s m^q  d  \sigma\lesssim (sb^{s-1})^s \norm{f}_{L^s(\sigma)}^s.\]
	By density and Fatou's lemma, we see that the above inequality holds for each $f
	\in L^s(\sigma)$.
	
	By duality, for $g\in L^q(m^q\sigma)$,
	\begin{align*}
		\norm{G^{(\alpha)}_{m^q\sigma} (g )}_{ L^q(\sigma)}&=\sup_{f\in L^s(\sigma)\setminus\{0\}}\frac{\int_M G^{(\alpha)}_{\sigma} (g m^q) f d\sigma }{\norm{f}_{L^s(\sigma)}}\\
		&=\sup_{f\in L^s(\sigma)\setminus\{0\}}\frac{\int_M  g G^{(\alpha)}_{\sigma}( f ) m^q d\sigma }{\norm{f}_{L^s(\sigma)}}\\
		&\le \sup_{f\in L^s(\sigma)\setminus\{0\}}\frac{\norm{g}_{ L^q(m^q\sigma)}  \left(\int_{M} (G^{(\alpha)}_{\sigma} (f))^s m^q  d  \sigma\right)^{1/s}}{\norm{f}_{L^s(\sigma)}}\\
		&\lesssim sb^{s-1} \norm{g}_{ L^q(m^q\sigma)}.
	\end{align*}
	Choose $g=1_{A}$ for some compact subset $A$ of $M$ with $0<\int_A m^q d\sigma <+\infty$. By Lemma \ref{lem-min}, we have  $m\lesssim G^{(\alpha)}_{m^q\sigma} (1_{A} )$, and hence
	\[ \norm{m}_{L^q(\sigma)}\lesssim  \norm{G^{(\alpha)}_{m^q\sigma} (1_{A} )}_{L^q(\sigma)}\lesssim sb^{s-1} \norm{1_{A}}_{ L^q(m^q\sigma)}
	\lesssim sb^{s-1}\left(\int_A m^q d\sigma\right)^{1/q}<+\infty.\]
	This finishes the proof of (\ref{mathsection-1}).
	
	To prove (\ref{mathsection-2}), first note that $r>a$ and for each $x\in M$,
	\begin{align*}
		\int_{\{y\in M:\,G^{(\alpha)}(o,y)>r^{-1}\}}G^{(\alpha)}(x,y) d  \sigma (y) &\le \int_{\{y\in M:\,G^{(\alpha)}(o,y)>r^{-1}\}}r^q G^{(\alpha)}(x,y)m(y)^q d  \sigma (y) \\
		&\le Cr^q m(x).
	\end{align*}
	Denote by $A_r=\{y\in M:\,G^{(\alpha)}(o,y)>r^{-1}\}$ and $\sigma_r=1_{A_r}\sigma$. We can rewrite the above estimate as
	\[G^{(\alpha)}_{\sigma_r} 1(x)\le C r^q m(x),\quad \mbox{for all $x\in M$}.\]
	Taking the $q$-th power of both sides and applying $G^{(\alpha)}_{\sigma_r}$, we obtain that
	\[G^{(\alpha)}_{\sigma_r}((G^{(\alpha)}_{\sigma_r}1)^q)\le C^q r^{q^2} G^{(\alpha)}_{\sigma_r}(m^q)\le C^{q+1}r^{q^2} m. \]
	We can then iterate this procedure and apply
	Corollary \ref{cor-psi=tq} to obtain for each $x\in M$ and $k\ge 1$,
	\begin{align*}
		[G^{(\alpha)}_{\sigma_r}1(x)]^{1+q+q^2+\cdots+q^k}\leq b^{q+q^2+\cdots+q^k}c(q,k) C^{1+q+q^2+\cdots+q^k} r^{q^{k+1}} m,
	\end{align*}
	where
	$c(q,k)=\prod_{j=1}^{k}(1+q+q^{2}+\cdots +q^{j})^{q^{k-j}}.$
	
	It follows that
	\begin{eqnarray}
		G^{(\alpha)}_{\sigma_r}1\leq C c(q,k)^{\frac{q-1}{q^{k+1}-1}}b^{\frac{q^{k+1}-q}{q^{k+1}-1}}r^{\frac{q^{k+1}(q-1)}{q^{k+1}-1}}m^{\frac{q-1}{q^{k+1}-1}}.
	\end{eqnarray}
	Letting $k\to\infty$, we have
	\begin{eqnarray}
		G^{(\alpha)}_{\sigma_r}1\leq c_1 r^{q-1},
	\end{eqnarray}
	for some constant $c_1>0$,
	where we have used that for each $k\ge 1$
	\begin{align*}
		c(q,k)^{\frac{q-1}{q^{k+1}-1}}&=\prod_{j=1}^{k}(1+q+\cdots+q^j)^{\frac{q^{k-j}(q-1)}{q^{k+1}-1}} \\
		&\leq\prod_{j=1}^{\infty}(1+q+\cdots+q^j)^{q^{-j}} \\
		&\leq\prod_{j=1}^{\infty}q^{jq^{-j}}\prod_{j=1}^{k}(1+q^{-1}+\cdots+q^{-j})^{q^{-j}} \\
		&\leq q^{-\frac{1}{q(q-1)^2}}\prod_{j=1}^{\infty}\left(\frac{q}{q-1}\right)^{q^{-j}} \\
		&\leq  q^{-\frac{1}{q(q-1)^2}}\left(\frac{q}{q-1}\right)^{\frac{1}{q-1}}.
	\end{align*}
	
\end{proof}
	
	\begin{proof}[Proof of the ``$\impliedby$'' direction]
		{First note that (\ref{mathsection-2}) implies
			that $G^{(\alpha)}(x, \cdot)
			$ is locally intergrable with respect to $\sigma$. Indeed, since $G^{(\alpha)}(\cdot, \cdot)$ is positive everywhere and lower semi-continuous, we have
			\[M=\bigcup_{1/a>t>0}\{z\in M:\,G^{(\alpha)}(o,z)>t\},\] as the union of a nested family of open sets.
			Then for each compact subset $A$ of $M$, there is some $t_A\in (0,1/a)$ such that $A\subseteq\{y\in M:\,G^{(\alpha)}(o,y)>t_A\}$. Clearly (\ref{mathsection-2}) implies that
			\[\int_A G^{(\alpha)}(x, y) d\sigma(y)<+\infty\] for each $x\in M$.}
		
		It follows that for all $x\in M$,
		\begin{align}\label{est-1}
			\int_M G^{(\alpha)}(x,y) m^q(y) d \sigma(y)&=\int_{0}^{1/a}q t^{q-1}
			\int_{ \{z\in M:\,G^{(\alpha)}(o,z)>t\}}G^{(\alpha)}(x,y) d  \sigma (y) d t\nonumber\\
			&\lesssim \int_{0}^{1/a}q t^{q-1}\cdot
			(\frac{1}{t})^{q-1} d t\nonumber\\
			&=q/a.
		\end{align}
		
		Now fix an arbitrary $x\in M$ with $G^{(\alpha)}(x,o)<1/a$. Combining with (\ref{mathsection-1}), we obtain
		\begin{align}\label{est-2}
			\int_{\{y\in M: G^{(\alpha)}(x,y)\le \kappa^2 G^{(\alpha)}(o, x)\}} G^{(\alpha)}(x,y) m^q(y) d \sigma(y)&\le \kappa^2 G^{(\alpha)}(o, x)\int_M m^q(y) d \sigma(y)\nonumber\\
			&\lesssim \kappa^2 G^{(\alpha)}(o, x).
		\end{align}
		On the other hand, if $y\in M$ satisfies $G^{(\alpha)}(x,y)> \kappa^2 G^{(\alpha)}(o, x)$,
		we then have
		\[\kappa G^{(\alpha)}(o,y)\ge G^{(\alpha)}(x,y)\wedge G^{(\alpha)}(o,x)=G^{(\alpha)}(o,x), \]
		and
		\[\kappa G^{(\alpha)}(o,x)\ge G^{(\alpha)}(x,y)\wedge G^{(\alpha)}(o,y)\ge (\kappa^2 G^{(\alpha)}(o,x))\wedge G^{(\alpha)}(o,y), \]
		which implies that
		\[\kappa G^{(\alpha)}(o,x)\ge G^{(\alpha)}(o,y).\]
		By the definition of $m(y)$, we obtain
		\begin{align}\label{est-3}
			&\int_{\{y\in M: G^{(\alpha)}(x,y)> \kappa^2 G^{(\alpha)}(o, x)\}} G^{(\alpha)}(x,y) m^q(y) d \sigma(y)\nonumber\\
			\lesssim &(\kappa G^{(\alpha)}(o, x))^q \int_{\{y\in M: G^{(\alpha)}(x,y)> \kappa^2 G^{(\alpha)}(o, x)\}} G^{(\alpha)}(x,y)  d \sigma(y)
			\nonumber\\
			\lesssim &(\kappa G^{(\alpha)}(o, x))^q \cdot ( (\kappa^2 G^{(\alpha)}(o, x)))^{1-q}\nonumber\\
			\lesssim&\kappa^{-q} G^{(\alpha)}(o, x).
		\end{align}
		Combining with (\ref{est-1})-(\ref{est-3}), we complete the proof.
	\end{proof}
	
	\section{Proofs of Theorem \ref{thm-main}, Theorem \ref{thm-main-mu} and Theorem \ref{thm-Henon}}\label{sect-proofs-vol}
	
	Under conditions (VD) and (PI), the Green kernel $G^{(\alpha)}$ for $(-\Delta)^{\alpha} $ satisfies the following estimate
	\[G^{(\alpha)}(x,y)\asymp
	\int_{d(x,y)}^{+\infty }\frac{t^{2\alpha-1}d t}{\mu (B(x,t))},\quad \mbox{for all $x,y\in
		M$}.\]
	In the following we closely follow the proofs in Section 6 of \cite{GSV}, with some changes to fix a minor inaccuracy there.
	\begin{lemma}\label{G-alpha-qm}\rm
		Fix $o\in M$, and for each $%
		\rho >0$, set
		\begin{equation}
			R(\rho ):=\int_{\rho }^{+\infty }\frac{t^{2\alpha-1}dt}{\mu (B(o,t))}.  \label{R}
		\end{equation}%
		Then the following holds:
		\begin{equation}
			R(\rho )\leq CR(2\rho ),\quad \forall \rho >0.  \label{R1}
		\end{equation}
		Furthermore, for some constant $\kappa>0$ and for each $x,y,z\in M$,
		\[R(d(x,z))\ge \kappa \min\{R(d(x,y)), R(d(y,z))\}.\]
		
	\end{lemma}
	\begin{proof}
		Note that $R(\rho )$ is strictly decreasing and continuous in $\rho $. By the
		doubling property (\ref{D}), letting $t=2s$ in (\ref{R}), we obtain
		\begin{equation*}
			R(2\rho )=\int_{2\rho }^{+\infty }\frac{t^{2\alpha-1}dt}{\mu (B(o,t))}=2^{2\alpha}\int_{\rho
			}^{+\infty }\frac{s^{2\alpha-1}ds}{\mu (B(o,2s))}\geq cR(\rho ).
		\end{equation*}
		
		Without loss of generality, we can assume that $x,y,z\in M$ are distinct. By the triangle inequality, either $d(x,y)\ge \frac{1}{2}d(x,z)$ or $d(y,z)\ge \frac{1}{2}d(x,z)$ holds.
		If $d(x,y)\ge \frac{1}{2}d(x,z)$, then
		\[R(d(x,y))\le R(\frac{1}{2}d(x,z)) \lesssim R(d(x,z)).\]
		Similarly if $d(y,z)\ge \frac{1}{2}d(x,z)$,
		then
		\[R(d(y,z))\le R(\frac{1}{2}d(x,z)) \lesssim  R(d(x,z)).\]
		Combining the above estimates, we complete the proof.
	\end{proof}
	\begin{proof}[Proof of Theorem \ref{thm-main}]
		We divide the proof into two steps.
		
		\textbf{Step 1.} The ``$\implies$'' direction:
		by Theorem \ref{thm-equiv-integrability}, if there exists a lower semi-continuous positive solution to (\ref{eq-non-local-integral-inequality}), then the following two inequalities hold:
		\begin{equation*}
			\int_{M}m(x)^{q} d  \sigma (x)<\infty , 
		\end{equation*}%
		and
		\begin{equation*}
			\sup_{x\in M}\int_{\{y\in M:\,G^{(\alpha)}(o,y)>r^{-1}\}}G^{(\alpha)}(x,y) d  \sigma (y)\lesssim r^{q-1},  
		\end{equation*}%
		for all $r>a$.	
		
		Applying \[G^{(\alpha)}(x,y)\asymp
		\int_{d(x,y)}^{+\infty }\frac{t^{2\alpha-1}d t}{\mu (B(x,t))},\quad\mbox{for all $x,y\in
			M$},\]
		we have $G^{(\alpha)}(o,\cdot)\asymp R(d(o,\cdot))$. Hence from (\ref{mathsection-1}) we deduce that
		\[\int_M \min\{R(d(o,x)),\frac{1}{a}\}^q d\sigma(x)<+\infty.\]
		Note that 
		\begin{align*}
			\int_M \min\{R(d(o,x)),\frac{1}{a}\}^q d\sigma(x)=q\int_{0}^{1/a}s^{q-1}\sigma\left(\{x\in M: R(d(o,x))>s\}\right) ds.
		\end{align*}
		Make a change of variables ``$s=R(\rho)$'', and set $\frac{1}{a}=R(\rho_0)$. Note that
		\[\{x\in M: R(d(o,x))>s\}=B(o,\rho).\]
		We have
		\begin{align*}
			q\int_{\rho_0}^{+\infty}\left(\int_{\rho}^{+\infty}\frac{t^{2\alpha -1
				}dt}{\mu(B(o,t))}\right)^{q-1}\frac{\sigma(B(o,\rho
				))}{\mu(B(o,\rho))}\rho^{2\alpha -1} dr=\int_M \min\{R(d(o,x)),\frac{1}{a}\}^q d\sigma(x)<+\infty.
		\end{align*}
		This proves (\ref{cond-int1}).
		
		Now we turn to (\ref{cond-int2}).
		Note that for some constant $c>0$, and for each $r>0$ and $y\in M$,
		\[R(d(o,y))>\frac{1}{cr} \implies G^{(\alpha)}(o,y)>\frac{1}{r},\]
		by the (VD) condition and estimates of $G^{(\alpha)}$.
		If $r>r_0$ for some $r_0$ sufficiently large, then by monotonicity and continuity, there is always a unique $\rho=\rho(r)>0$ such that
		$R(\rho)=\frac{1}{cr}$. Hence,
		\begin{equation*}
			R\left( d\left( o,y\right) \right)
			>R\left( \rho \right) \iff d\left( o,y\right) <\rho \implies	G^{(\alpha)}\left( o,y\right) >r^{-1}.
		\end{equation*}%
		Then for $\rho$ with $R(\rho)=\frac{1}{cr}$, we have
		\[B(o,\rho)\subseteq \{y\in M: 	G^{(\alpha)}\left( o,y\right) >r^{-1}\}.\]
		We can then estimate for $r>r_0$,
		\begin{align*}
			\int_{\{y\in M:\,G^{(\alpha)}(o,y)>r^{-1}\}}G^{(\alpha)}(x,y) d  \sigma (y)&\ge c\int_{B(o,\rho)}\int_{d(x,y)}^{+\infty }\frac{t^{2\alpha-1}d t}{\mu (B(x,t))} d \sigma(y)
			\\ &=c\int_{0}^{+\infty} \frac{\sigma(B(x,t)\cap B(o,\rho))}{\mu(B(x,t))}
			t^{2\alpha-1} dt.
		\end{align*}
		Hence for each $r>a$,
		\begin{align*}
			Cc^{1-q}R(\rho)^{1-q}&= C r^{q-1}\\
			&\ge \int_{\{y\in M:\,G^{(\alpha)}(o,y)>r^{-1}\}}G^{(\alpha)}(x,y) d  \sigma (y)\\
			&\ge c \int_{0}^{+\infty} \frac{\sigma(B(x,t)\cap B(o,\rho))}{\mu(B(x,t))}
			t^{2\alpha-1} dt.
		\end{align*}
		Then we have
		\begin{align*}
			\sup_{x\in M, \rho>\rho(r_0)}\left[ \int_{0}^{+\infty}\frac{\sigma (B(x,s)\cap B(o, \rho))}{\mu (B(x,s))}%
			\,s^{2\alpha-1} d  s\right] \,\left[ \int_{\rho}^{+\infty }\frac{t^{2\alpha-1} d  t}{\mu (B(o,t))}\right]
			^{q-1}<+\infty.
		\end{align*}
		
		
		\textbf{Step 2.} The ``$\impliedby$'' direction:
		since
		\begin{equation*}
			\int_{r_{0}}^{+\infty }\left[ \int_{r}^{+\infty }\frac{t^{2\alpha-1} d  t}{\mu (B(o,t))}%
			\right] ^{q-1}\frac{\sigma (B(o,r))}{\mu (B(o,r))}r^{2\alpha-1} d  r<\infty ,
		\end{equation*}%
		and
		\begin{align*}
			\sup_{x\in M, r>r_0}\left[ \int_{0}^{+\infty}\frac{\sigma (B(x,s)\cap B(o, r))}{\mu (B(x,s))}%
			\,s^{2\alpha-1} d  s\right] \,\left[ \int_{r}^{+\infty }\frac{t^{2\alpha-1} d  t}{\mu (B(o,t))}\right]
			^{q-1}<+\infty.
		\end{align*}
		
		It is sufficient to verify
		\begin{equation*}
			\int_{M}m(x)^{q} d  \sigma (x)<\infty ,
		\end{equation*}%
		and
		\begin{equation*}
			\sup_{x\in M}\int_{\{y\in M:\,G^{(\alpha)}(o,y)>r^{-1}\}}G^{(\alpha)}(x,y) d  \sigma (y)\lesssim
			r^{q-1},
		\end{equation*}%
		for all $r>a$.	
		
		As is just shown in step 1, using the same notations there, we have
		\begin{align*}
			\int_M \min\{R(d(o,x)),\frac{1}{a}\}^q d\sigma(x)=	q\int_{r_0}^{+\infty}(\int_{r}^{+\infty}\frac{t^{2\alpha -1
				}dt}{\mu(B(o,t))})^{q-1}\frac{\sigma(B(o,r
				))}{\mu(B(o,r))}r^{2\alpha -1} dr<+\infty.
		\end{align*}
		
		Similarly as in step 1, for some constant $c'>0$, there is always a unique $\rho=\rho(r)>0$ such that
		$R(\rho)=\frac{1}{c'r}$. Hence,
		\begin{equation*}
			R\left( d\left( o,y\right) \right)
			>R\left( \rho \right) \iff d\left( o,y\right) <\rho \impliedby	G^{(\alpha)}\left( o,y\right) >r^{-1}.
		\end{equation*}
		
		Then for $\rho$ with $R(\rho)=\frac{1}{c'r}$, we have
		\[B(o,\rho)\supseteq \{y\in M: 	G^{(\alpha)}\left( o,y\right) >r^{-1}\}.\]
		We can then estimate for $r>r_0$,
		\begin{align*}
			\int_{\{y\in M:\,G^{(\alpha)}(o,y)>r^{-1}\}}G^{(\alpha)}(x,y) d  \sigma (y)&\le c'\int_{B(o,\rho)}\int_{d(x,y)}^{+\infty }\frac{t^{2\alpha-1}d t}{\mu (B(x,t))} d \sigma(y)
			\\ &=c'\int_{0}^{+\infty} \frac{\sigma(B(x,t)\cap B(o,\rho))}{\mu(B(x,t))}
			t^{2\alpha-1} dt\\ &\lesssim  R(\rho)^{1-q}\\ &\lesssim  c'^{q-1} r^{q-1}.
		\end{align*}
		Hence, we complete the proof.
	\end{proof}
	Now we turn to Theorem \ref{thm-main-mu}. If $\sigma=\mu$, the conditions (\ref{cond-int1}) and (\ref{cond-int2}) amount to respectively
	\begin{equation}
		\int_{r_{0}}^{+\infty }\left[ \int_{r}^{+\infty }\frac{t^{2\alpha-1} d  t}{\mu (B(o,t))}%
		\right] ^{q-1}r^{2\alpha-1} d  r<\infty ,
		\label{cond-int1-mu}
	\end{equation}%
	and
	\begin{equation}
		\sup_{x\in M, r>r_0}\left[ \int_{0}^{+\infty}\frac{\mu (B(x,s)\cap B(o, r))}{\mu (B(x,s))}%
		\,s^{2\alpha-1} d  s\right] \,\left[ \int_{r}^{+\infty }\frac{t^{2\alpha-1} d  t}{\mu (B(o,t))}\right]
		^{q-1}<+\infty.
		\label{cond-int2-mu}
	\end{equation}%
	\begin{proposition}\label{prop-imply}\rm
		Assume that the fractional Laplacian $	(-\Delta)^{\alpha}$ is transient and the estimates on the Green function
		\[G^{(\alpha)}(x,y)\asymp
		\int_{d(x,y)}^{+\infty }\frac{t^{2\alpha-1}d t}{\mu (B(x,t))},\quad x,y\in
		M.\]
		Then (\ref{cond-int2-mu}) follows from (\ref{cond-int1-mu}).
	\end{proposition}
	We first show an elementary lemma.
	\begin{lemma}\label{lem-elementary}\rm
		Let $\{a_k\}_{k\ge 1}, \{u_k\}_{k\ge 1}$ be two sequences of positive numbers. Assume that $\{a_k\}_{k\ge 1}$ is non-decreasing, and that $\{u_k\}_{k\ge 1}$ is non-increasing with $\lim\limits_{k\rightarrow\infty} u_k =0$. If
		\[\sum_{k=1}^{\infty} a_k(u_k-u_{k+1})<+\infty,\]
		then  $\lim\limits_{k\rightarrow\infty} a_k u_k =0$. In particular, $\{a_k u_k\}_{k\ge 1}$ is bounded.
	\end{lemma}
	\begin{proof}
		For all $m>k\ge 1$, we have
		\[\sum_{l=k}^{m} a_l(u_l-u_{l+1})\ge a_k (u_{k}-u_{m+1}).\]
		It follows that
		\[a_k u_k \le\sum_{l=k}^{m} a_l(u_l-u_{l+1})+ a_k u_{m+1}.\]
		Letting $m\rightarrow \infty$, we have
		\[a_k u_k \le\sum_{l=k}^{\infty} a_l(u_l-u_{l+1}).\]
		Then $\lim\limits_{k\rightarrow\infty} a_k u_k =0$.
	\end{proof}
	
	\begin{proof}[Proof of Proposition \ref{prop-imply}]
		Assume the transience of $G^{(\alpha)}$, that is
		\[\int_{1}^{+\infty }\frac{t^{2\alpha-1} d  t}{\mu (B(o,t))}<+\infty,\]
		and assume also
		\begin{equation*}
			\int_{r_{0}}^{+\infty }\left[ \int_{r}^{+\infty }\frac{t^{2\alpha-1} d  t}{\mu (B(o,t))}%
			\right] ^{q-1}r^{2\alpha-1} d  r<\infty.
		\end{equation*}
		Let us show
		 \[
			\sup_{x\in M,r>r_0}\left[ \int_{0}^{+\infty}\frac{\mu (B(x,s)\cap B(o, r))}{\mu (B(x,s))}%
			\,s^{2\alpha-1} d  s\right] \,\left[ \int_{r}^{+\infty }\frac{t^{2\alpha-1} d  t}{\mu (B(o,t))}\right]
			^{q-1}<\infty.\]
		We consider two cases:
		\begin{enumerate}
			\item $d(x, o)\le 2r$;
			\item $d(x, o)> 2r$.
		\end{enumerate}
		Note that in the first case, $B(x, s)\supseteq B(o,r)$ if $s> 3r$.
		We have
		\begin{align*}
			&\int_{0}^{+\infty}\frac{\mu (B(x,s)\cap B(o, r))}{\mu (B(x,s))}  s^{2\alpha-1} d  s\\
			\le & \int_{0}^{3r} s^{2\alpha-1} d  s + \int_{3r}^{+\infty}\frac{\mu ( B(o, r))}{\mu (B(x,s))}  s^{2\alpha-1} d  s.
		\end{align*}
		In the second summand above, $d(x,o)\le 2r< s$. Since $B(x, s)\subseteq B(o, 2s)$ and $B(o, s)\subseteq B(x, 2s)$, and by (VD),  we derive
		\[\mu (B(x,s))\asymp\mu (B(o,s)).\]
		Hence
		\begin{align}\label{bdd-0}
		\int_{0}^{+\infty}\frac{\mu (B(x,s)\cap B(o, r))}{\mu (B(x,s))}  s^{2\alpha-1} d  s \lesssim r^{2\alpha} +\mu ( B(o, r))\int_{3r}^{+\infty}\frac{1}{\mu (B(o,s))}  s^{2\alpha-1} d  s.
		\end{align}
		
		By the doubling property of the quantities involved, we have
		\begin{align*}
			&\sum_{k=1}^{\infty}(2^k r_0)^{2\alpha}\left[\sum_{l=k}^{\infty}\frac{(2^l r_0)^{2\alpha}}{\mu(B(o, 2^l r_0))}\right]^{q-1} \\
			\asymp	&\int_{r_{0}}^{+\infty }\left[ \int_{r}^{+\infty }\frac{t^{2\alpha-1} d  t}{\mu (B(o,t))}%
			\right] ^{q-1}r^{2\alpha-1} d  r<\infty.
		\end{align*}
		It is then direct to see that
		\begin{equation}\label{bdd-1}
			r^{2\alpha}\left[ \int_{r}^{+\infty }\frac{t^{2\alpha-1} d  t}{\mu (B(o,t))}%
			\right] ^{q-1}<+\infty.
		\end{equation}
		For $2^k r_0 \le r\le 2^{k+1} r_0$, we have
		\begin{align*}
			&\mu(B(o,r))\int_{3r}^{+\infty}\frac{1}{\mu (B(o,s))}  s^{2\alpha-1} d  s\left[ \int_{r}^{+\infty }\frac{t^{2\alpha-1} d  t}{\mu (B(o,t))}%
			\right] ^{q-1}\\ \asymp	&\mu(B(o,r))\left[ \int_{r}^{+\infty }\frac{t^{2\alpha-1} d  t}{\mu (B(o,t))}%
			\right] ^{q}\\
			\asymp &\mu(B(o, 2^k r_0))\left[\sum_{l=k}^{\infty}\frac{(2^l r_0)^{2\alpha}}{\mu(B(o, 2^l r_0))}\right]^{q}.
		\end{align*}
		Setting
		\[a_k =\mu(B(o, 2^k r_0)),~~~~ v_k = \sum_{l=k}^{\infty}\frac{(2^l r_0)^{2\alpha}}{\mu(B(o, 2^l r_0))}.\]
		Then we have
		\begin{align*}
			+\infty>& q\sum_{k=1}^{\infty} (2^k r_0)^{2\alpha} v_k^{q-1}\\
			=&q\sum_{k=1}^{\infty} a_k v_k^{q-1}(v_k -v_{k+1})\\\ge &\sum_{k=1}^{\infty} a_k (v_k^q -v_{k+1}^q).
		\end{align*}
		By Lemma \ref{lem-elementary}, we see that $\{a_k v_k^q\}_{k\ge 1}$  is bounded, and hence
		\begin{equation}\label{bdd-2}
			\mu(B(o,r))\left[ \int_{r}^{+\infty }\frac{t^{2\alpha-1} d  t}{\mu (B(o,t))}
			\right] ^{q}<+\infty.
		\end{equation}
		It follows from (\ref{bdd-0}), (\ref{bdd-1}), and  (\ref{bdd-2}) that
		\[
		\left[ \int_{0}^{+\infty}\frac{\mu (B(x,s)\cap B(o, r))}{\mu (B(x,s))}%
		\,s^{2\alpha-1} d  s\right] \,\left[ \int_{r}^{+\infty }\frac{t^{2\alpha-1} d  t}{\mu (B(o,t))}\right]
		^{q-1}\] is bounded.
		
		Now we consider the second case. Since $d(x,o)>2r$, $B(x,s)\cap B(o,r)=\emptyset$ if $s<d(x,o)-r$. Hence, we can estimate
		\begin{align*}
			&\int_{0}^{+\infty}\frac{\mu (B(x,s)\cap B(o, r))}{\mu (B(x,s))}  s^{2\alpha-1} d  s\\
			\le &  \int_{d(x,o)-r}^{+\infty}\frac{\mu ( B(o, r))}{\mu (B(x,s))}  s^{2\alpha-1} d  s\\
			\lesssim &  \int_{\frac{1}{2} d(x,o)}^{+\infty}\frac{\mu ( B(o, r))}{\mu (B(o,s))}  s^{2\alpha-1} d  s\\ \lesssim
			&   \mu ( B(o, r)) \int_{r}^{+\infty}\frac{ s^{2\alpha-1}}{\mu (B(o,s))}  d  s.
		\end{align*}
		Applying the same argument as in the first case, we obtain that \[
		\mu(B(o,r))\left[ \int_{r}^{+\infty }\frac{t^{2\alpha-1} d  t}{\mu (B(o,t))}%
		\right] ^{q}\] is bounded. Hence \[
		\left[ \int_{0}^{+\infty}\frac{\mu (B(x,s)\cap B(o, r))}{\mu (B(x,s))}%
		\,s^{2\alpha-1} d  s\right] \,\left[ \int_{r}^{+\infty }\frac{t^{2\alpha-1} d  t}{\mu (B(o,t))}\right]
		^{q-1}\] is bounded as well. 
	\end{proof}

	For the proof of Theorem \ref{thm-main-mu},
	we need the following lemma, which is a simple variant of  \cite[Lemma 6.2]{GSV}.
	
	\begin{lemma}\label{prop-s}\rm Let $s\in (0,1)$, and let $\phi :\,(0,+\infty )\rightarrow
		(0,+\infty )$ be a non-increasing function. Then there exists a positive
		constant $C=C(s)$ such that, for all $r>0$,
		\begin{equation}
			\left( \int_{r}^{\infty }\phi (t)\,t^{2\alpha-1}\,dt\right) ^{s}\leq C\int_{r}^{\infty
			}\phi (t)^{s}\,t^{2\alpha s-1}\,dt+Cr^{2\alpha s}\,\phi (r)^{s}.  \label{est-decr}
		\end{equation}
	\end{lemma}
	
	\begin{proof}
		We have
		\begin{align*}
			\left( \int_{r}^{\infty }\phi (t)\,t^{2\alpha-1}\,dt\right) ^{s}& =s\,\int_{r}^{\infty
			}\left( \int_{r}^{t}\phi (\tau )\,\tau^{2\alpha-1} \,d\tau \right) ^{s-1}\phi (t)\,t^{2\alpha-1}\,dt
			\\
			& \leq s\,\int_{r}^{\infty }\left( \int_{r}^{t}\,\tau^{2\alpha-1} \,d\tau \right)
			^{s-1}\phi (t)^{s}\,t^{2\alpha-1}\,dt \\
			& =s\,(2\alpha)^{1-s}\,\int_{r}^{\infty }\left( t^{2\alpha}-r^{2\alpha}\right) ^{s-1}\phi
			(t)^{s}\,t^{2\alpha-1}\,dt \\
			& =s\,(2\alpha)^{1-s}\,(I_{1}+I_{2}),
		\end{align*}%
		where
		\begin{equation*}
			I_{1}=\int_{2r}^{\infty }\left( t^{2\alpha}-r^{2\alpha}\right) ^{s-1}\phi
			(t)^{s}\,t^{2\alpha-1}\,dt,\quad I_{2}=\int_{r}^{2r}\left( t^{2\alpha}-r^{2\alpha}\right) ^{s-1}\phi
			(t)^{s}t^{2\alpha-1}\,dt.
		\end{equation*}%
		Clearly, for $t>2r$,
		\begin{equation*}
			\left( t^{2\alpha}-r^{2\alpha}\right) ^{s-1}\leq \left(1- 2^{-2\alpha}\right)
			^{s-1}\,t^{2\alpha(s-1)},
		\end{equation*}%
		whence
		\begin{equation*}
			I_{1}\leq \left(1- 2^{-2\alpha}\right)
			^{s-1}\,\int_{2r}^{\infty }\phi
			(t)^{s}\,t^{2\alpha s-1}\,dt.
		\end{equation*}%
		On the other hand, the change $\xi =t^{2\alpha}-r^{2\alpha}$ yields
		\begin{equation*}
			\int_{r}^{2r}\left( t^{2\alpha}-r^{2\alpha}\right) ^{s-1}t^{2\alpha-1}\,dt=\frac{1}{2\alpha}%
			\int_{0}^{(2^{2\alpha}-1)r^{2\alpha}}\xi ^{s-1}d\xi =\frac{(2^{2\alpha}-1)^s}{2\alpha s}r^{2\alpha s},
		\end{equation*}%
		whence
		\begin{equation*}
			I_{2}\leq \phi (r)^{s}\int_{r}^{2r}\left( t^{2\alpha}-r^{2\alpha}\right) ^{s-1}t^{2\alpha-1}\,dt=%
			\frac{(2^{2\alpha}-1)^s}{2\alpha s}\phi (r)^{s}r^{2\alpha s}.
		\end{equation*}%
		Combining the estimates of $I_{1}$ and $I_{2}$, we deduce (\ref{est-decr}).
	\end{proof}
To finally prove Theorem \ref{thm-main-mu}, by Proposition \ref{prop-imply}, 
we only need to show that the condition
		\begin{equation}
				\int_{r_{0}}^{+\infty }\left[ \int_{r}^{+\infty }\frac{t^{2\alpha-1}dt}{\mu (B(o,t))}%
				\right] ^{q-1}r^{2\alpha-1}dr<\infty ,  \label{cond-int1ab}
			\end{equation}%
		is equivalent to the simpler condition
		\begin{equation}
				\int_{r_{0}}^{+\infty }\frac{r^{2\alpha q-1}dr}{[\mu (B(o,r))]^{q-1}}<\infty .
				\label{cond-int1bb}
			\end{equation}
		  We omit the details and refer to \cite[Section 6]{GSV}.

Theorem \ref{thm-Henon} can be proved by arguments similar to those for Proposition \ref{prop-imply}. In the setting of Theorem \ref{thm-Henon}, $\sigma =\lvert \cdot\rvert^{\gamma} \mu$ with $\mu$ being the Lebesgue measure on $\mathbb{R}^n$. We need show that conditions (\ref{cond-int1}) and (\ref{cond-int2}) hold if and only if $q>\frac{n+\gamma}{n-2\alpha}$. 

It is clear that $\sigma(B(o,r))\asymp r^{n+\gamma}$, and hence condition (\ref{cond-int1}) amounts to
\[	\int_{r_{0}}^{+\infty }\left[ \int_{r}^{+\infty }t^{2\alpha-1-n} d  t%
\right] ^{q-1}r^{\gamma+2\alpha-1} d  r<\infty.\]
It can simplified as 
\[\int_{r_{0}}^{+\infty } r^{(2\alpha-n)(q-1)+\gamma+2\alpha-1}d r<\infty,\]
which is equivalent to 
\[q>\frac{n+\gamma}{n-2\alpha}.\]
\begin{proof}[Proof of Theorem \ref{thm-Henon}]
By the above preliminary discussions,  condition (\ref{cond-int1}) holds if and only if $q> \frac{n+\gamma}{n-2\alpha}$.
We only need show the following claim:

\textbf{Claim}: condition (\ref{cond-int2}) holds if and only if $q\ge \frac{n+\gamma}{n-2\alpha}$.

To estimate
\[
\sup_{x\in M,r>r_0}\left[ \int_{0}^{+\infty}\frac{\sigma (B(x,s)\cap B(o, r))}{\mu (B(x,s))}%
\,s^{2\alpha-1} d  s\right] \,\left[ \int_{r}^{+\infty }\frac{t^{2\alpha-1} d  t}{\mu (B(o,t))}\right]
^{q-1},\]
we consider two cases:
\begin{enumerate}
	\item $d(x, o)\le 2r$;
	\item $d(x, o)> 2r$.
\end{enumerate}

First suppose $d(x,o)\le 2r$.  Note that $B(x, s)\supseteq B(o,r)$ if $s> 3r$. We have
\begin{align*}
&\int_{0}^{+\infty}\frac{\sigma (B(x,s)\cap B(o, r))}{\mu (B(x,s))}%
\,s^{2\alpha-1} d  s\\
=&\int_{0}^{3r}\frac{\sigma (B(x,s)\cap B(o, r))}{\mu (B(x,s))}%
\,s^{2\alpha-1} d  s+\int_{3r}^{+\infty}\frac{\sigma (B(x,s)\cap B(o, r))}{\mu (B(x,s))}%
\,s^{2\alpha-1} d  s\\
\le &\int_{0}^{3r}\frac{\sigma (B(x,s))}{\mu (B(x,s))}%
\,s^{2\alpha-1} d  s+\int_{3r}^{+\infty}\frac{\sigma (B(o, r))}{\mu (B(x,s))}%
\,s^{2\alpha-1} d  s \\
 \lesssim &\int_{0}^{3r}\frac{s^{n+\gamma}}{s^n}%
 \,s^{2\alpha-1} d  s+\int_{3r}^{+\infty}\frac{r^{n+\gamma}}{s^n}%
 \,s^{2\alpha-1} d  s\\
 \lesssim & r^{\gamma+ 2\alpha}+r^{n+\gamma}\cdot r^{2\alpha-n}\\
 \asymp & r^{\gamma+ 2\alpha}, 
\end{align*}
where the constants involved in the steps with ``$\lesssim$" and ``$\asymp$" only depend on $n$. Also note that $\gamma>-2\alpha$ is assumed.

On the other hand, we have that
\begin{align*}
	&\int_{0}^{+\infty}\frac{\sigma (B(x,s)\cap B(o, r))}{\mu (B(x,s))}%
	\,s^{2\alpha-1} d  s\\
	\ge &\int_{3r}^{+\infty}\frac{\sigma (B(x,s)\cap B(o, r))}{\mu (B(x,s))}%
	\,s^{2\alpha-1} d  s\\
=&\int_{3r}^{+\infty}\frac{\sigma (B(o, r))}{\mu (B(x,s))}%
	\,s^{2\alpha-1} d  s \\
\asymp &\int_{3r}^{+\infty}\frac{r^{n+\gamma}}{s^n}%
	\,s^{2\alpha-1} d  s\\
	\asymp & r^{\gamma+ 2\alpha}.
\end{align*}
Consequently, we obtain 
\begin{align*}
&\left[ \int_{0}^{+\infty}\frac{\sigma (B(x,s)\cap B(o, r))}{\mu (B(x,s))}%
\,s^{2\alpha-1} d  s\right] \,\left[ \int_{r}^{+\infty }\frac{t^{2\alpha-1} d  t}{\mu (B(o,t))}\right]
^{q-1}\\
\asymp &  r^{(\gamma+ 2\alpha)+(2\alpha-n)(q-1)},
\end{align*}
if $ d(x,o)\le 2r$.

Now suppose $ d(x,o)> 2r$. Note that $B(x,s)\cap B(o,r)=\emptyset$ if $s<d(x,o)-r$. 

We have
	\begin{align*}
	&\int_{0}^{+\infty}\frac{\sigma (B(x,s)\cap B(o, r))}{\mu (B(x,s))}  s^{2\alpha-1} d  s\\
	\le &  \int_{d(x,o)-r}^{+\infty}\frac{\sigma ( B(o, r))}{\mu (B(x,s))}  s^{2\alpha-1} d  s\\
\le &  \int_{\frac{1}{2} d(x,o)}^{+\infty}\frac{\sigma ( B(o, r))}{\mu (B(o,s))}  s^{2\alpha-1} d  s\\ \lesssim
	&   r^{n+\gamma}\cdot d(x,o)^{2\alpha-n}.
\end{align*}
For the other direction, we observe that $B(x, s)\supseteq B(o,r)$ if $s> \frac{3}{2}d(x,o)$.
It follows that
\begin{align*}
	&\int_{0}^{+\infty}\frac{\sigma (B(x,s)\cap B(o, r))}{\mu (B(x,s))}%
	\,s^{2\alpha-1} d  s\\
	\ge &\int_{\frac{3}{2}d(x,o)}^{+\infty}\frac{\sigma (B(x,s)\cap B(o, r))}{\mu (B(x,s))}%
	\,s^{2\alpha-1} d  s\\
	=&\int_{\frac{3}{2}d(x,o)}^{+\infty}\frac{\sigma (B(o, r))}{\mu (B(x,s))}%
	\,s^{2\alpha-1} d  s \\
	\asymp &  r^{n+\gamma}\cdot d(x,o)^{2\alpha-n}.
\end{align*}
Hence we have
\begin{align*}
	&\left[ \int_{0}^{+\infty}\frac{\sigma (B(x,s)\cap B(o, r))}{\mu (B(x,s))}%
	\,s^{2\alpha-1} d  s\right] \,\left[ \int_{r}^{+\infty }\frac{t^{2\alpha-1} d  t}{\mu (B(o,t))}\right]
	^{q-1}\\
	\asymp &    r^{(n+\gamma)+(2\alpha-n)(q-1)}\cdot d(x,o)^{2\alpha-n},
\end{align*}
if $ d(x,o)>2r$.

Combing the two cases together, we obtain that 
\[\sup_{x\in M,r>r_0}\left[ \int_{0}^{+\infty}\frac{\sigma (B(x,s)\cap B(o, r))}{\mu (B(x,s))}%
\,s^{2\alpha-1} d  s\right] \,\left[ \int_{r}^{+\infty }\frac{t^{2\alpha-1} d  t}{\mu (B(o,t))}\right]
^{q-1}<+\infty \] if and only if
\[(\gamma+ 2\alpha)+(2\alpha-n)(q-1) \le 0.\]
\end{proof}

	\section{Weak solutions for the fractional differential inequality}\label{sect-solutions}
So far we focused on the existence of solutions to the integral type inequality \eqref{eq-non-local-integral-inequality}. In this section, we prove Theorem \ref{thm-equivalence}, which gives a weak type solution for the differential type inequality \eqref{eq-non-local-differential-inequality}. For the sake of completeness, we summarize briefly some necessary facts from Dirichlet form theory in Appendix \ref{appendix-DF}.

\begin{proof}[Proof of Theorem \ref{thm-equivalence}]
	``if'' part: Suppose that $v$ is a   positive solution to (\ref{eq-non-local-differential-inequality}) in $\mathcal{F}^{(\alpha)}_e$-sense. By definition, $v\in\mathcal{F}^{(\alpha)}_e$, and
	\begin{equation}
		\mathcal{E}^{(\alpha)}(v, \varphi)\ge \int_{M} v^q \varphi d \sigma, \label{blacksquare}
	\end{equation}
	for each $\varphi\in \mathcal{F}^{(\alpha)}\cap C_c(M)$, with $\varphi\ge 0$.
	In particular, we see that $\mathcal{E}^{(\alpha)}(v, \varphi)\ge 0$ for each $\varphi\in \mathcal{F}^{(\alpha)}\cap C_c(M)$, with $\varphi\ge 0$. By  \cite[Lemma 2.2.10]{FOT}, there is  a Radon measure  $\nu$ of finite ($0$-order) energy integral (i.e. for some $C>0$, $\forall f\in \mathcal{F}^{(\alpha)}\cap C_c(M)$, $\int_{M}|f|
		d \nu \le \sqrt{\mathcal{E}^{(\alpha)}(f, f)}$), such that
		\[\mathcal{E}^{(\alpha)}(v, \varphi)=\int_M \tilde{\varphi} d \nu, \]for each $\varphi\in \mathcal{F}^{(\alpha)}_e$(here $\tilde{\varphi}$ is a quasi-continuous version of $\varphi$); see \cite[Theorem 2.1.7]{FOT}.
	Note that $\mathcal{F}^{(\alpha)}\cap C_c(M)$ is a special standard core (\cite[Section 1.1, Section 1.4]{FOT}). In particular, for each compact set $K$ and relatively
		compact open set $G$ with $K\subset G$, there is some $\varphi\in \mathcal{F}^{(\alpha)}\cap C_c(M)$, such that $0\le \varphi\le 1$, $\varphi =1$ on $K$, and $\varphi=0$ on $M\setminus G$. Then the inequality \[+\infty>\mathcal{E}^{(\alpha)}(v, \varphi)=\int_{M} \varphi d \nu \ge \int_{M}\varphi  v^q  d \sigma\ge\int_{K} v^q  d \sigma\] implies that $v^q\sigma$ is a Radon measure. And by \cite[Lemma 1.4.2(ii)]{FOT}, we have 
		\[\int_{M} \varphi d \nu \ge \int_{M} \varphi v^q  d \sigma\]
		for all $\varphi\in C_c(M)$ with $\varphi\ge 0$, and hence
		$\nu \ge v^q\sigma$  in the sense of measures. 
	
	Let $g$ be a reference function for $(\mathcal{E}^{(\alpha)}, \mathcal{F}^{(\alpha)})$, that is, $g$ is a strictly positive, bounded, $\mu$-integrable function  such that $\int_M g G^{(\alpha)} g d\mu \le 1$ (see \cite[ Section 1.5]{FOT} for more details). 
	Let $h$ be a bounded 
	measurable function and $f=hg$. 	By   \cite[Theorem 1.5.4 and Theorem 1.5.5]{FOT}, we have $G^{(\alpha)}f\in \mathcal{F}^{(\alpha)}_e$, and
		\begin{align*}
			\int_M v h g d \mu=&	\int_M v f d \mu \\= &\mathcal{E}^{(\alpha)}(v, G^{(\alpha)}f)\\
			=&\int_M G^{(\alpha)}f d \nu\\
			=&\int_M f G^{(\alpha)}_{\nu} 1 d\mu\\=&\int_M h (G^{(\alpha)}_{\nu} 1) gd\mu,
		\end{align*}
		where for the second last equality we applied the symmetry of the Green kernel.  
		
		We see that $v(x)=G^{(\alpha)}_{\nu} 1(x)$ for $g\mu$-a.e. $x\in M$ and hence for $\mu$-a.e. $x$. 
		Recalling that $\nu \ge v^q\sigma$,		we have
		\[G^{(\alpha)}_{\nu} 1(x)\ge G^{(\alpha)}_{\sigma}(v^q)(x),\] for each $x\in M$. 
		It follows that	$v\ge G^{(\alpha)}_{\sigma}(v^q)$ holds $\mu$ (and hence $\sigma$) almost everywhere. It is also clear that $G^{(\alpha)}_{\nu} 1$ is lower semi-continuous by Fatou's lemma. 
		Hence $v$ is a positive solution to  (\ref{eq-non-local-integral-inequality}) according to Definition \ref{defi-int-ineq}.
	
	
	
	
	``only if'' part: Now suppose that there exists a positive solution to (\ref{eq-non-local-integral-inequality}).
	
	Under the quasi-metric condition for $G^{(\alpha)}$, this is equivalent to
	\begin{equation}
		m \ge C G^{(\alpha)}_{\sigma}(m^q), \label{blacklozenge}
	\end{equation}
	for some constant $C>0$ by Theorem \ref{thm-equiv-minimal}.
	
	Rescaling in a similar way as  in the proof of ``only if'' part of Theorem \ref{thm-equiv-minimal} (the equation (\ref{blacktriangle})), we see that for some $c>0$, and a non-negative, non-zero function $\eta \in C_c(M)$ with $\eta \le c' m^q$,
	\[cm\ge G^{(\alpha)}_{\sigma} ((cm)^q)+ G^{(\alpha)}_{\sigma}(\eta).\]
	Here the existence of $\eta$ follows from the lower semi-continuity of $m$.
	
	By assumption $\theta=\frac{d\sigma}{d\mu}\in C(M)$. Note that $m\in L^q(M, \sigma)$ by Theorem \ref{thm-equiv-integrability}. We have $m^q \theta \in L^1(M, \mu)$. By the inequality (\ref{blacklozenge}), it follows that
	
		\begin{align*}
			&\int_M G^{(\alpha)}(m^q \theta) (m^q \theta) d\mu
			\\=&\int_M m^q \theta G^{(\alpha)}_{\sigma}(m^q) d\mu
			\\
			\le &\int_M Cm \cdot  m^q d\sigma\\
			\le &\int_M C' m^q d\sigma<+\infty.
	\end{align*}
	A non-negative measurable function $f$ is said to be of finite energy, if  $\int_M f G^{(\alpha)} f d\mu <+\infty$.
		The above estimates show that $m^q\theta$ is a function of finite energy.
	Consider the following iteration
	\begin{align*}
		\begin{cases}
			v_0= 0;\\
			v_{k+1}=G^{(\alpha)} (v_{k}^q\theta)+G^{(\alpha)} (\eta\theta).
		\end{cases}
	\end{align*}
	And let $v=\sup_{k\ge 0} v_k$. It is clear by induction that $\{v_k\}_{k\ge 0}$ increases to $v$, $v\le cm$, and
	\[v=G^{(\alpha)} (v^q\theta)+G^{(\alpha)} (\eta\theta)\]
	holds. Note that $v^q \theta$ and  $\eta \theta$, both bounded by $cm^q \theta$, are functions of finite energy. Then by  \cite[ Theorem 1.5.4]{FOT}, $v\in \mathcal{F}^{(\alpha)}_e$, and
	\[\mathcal{E}^{(\alpha)}(v, \varphi)=\int_M v^q \theta \varphi d\mu+\int_M \eta \theta \varphi d\mu\ge \int_M v^q \varphi d\sigma,\]
	for each $\varphi\in \mathcal{F}^{(\alpha)}\cap C_c(M)$, with $\varphi\ge 0$.

\end{proof}

	\section{Strong solutions for the fractional differential inequality}\label{sect-strong-solutions}

	Based on Theorem \ref{thm-equivalence}, we can lift the regularity of solutions further, in a similar way as \cite{GSV}. The argument is a bit involved due to our lack of knowledge for the regularity of fractional Laplacian on manifolds. 
	
	
	Let $M$ be a complete connected non-compact Riemannian manifold. 
	Slightly abusing the notations, we denote this semigroup on $C_{\infty}(M)$ by $\{P_t\}_{t\ge 0}$ as well, since they are both determined by the heat kernel  $\{p_{t}(x,y)\}_{t> 0}$.  Fix some $\alpha\in (0,1)$. The subordination construction applied to  $\{P_t\}_{t\ge 0}$ on $C_{\infty}(M)$ leads to a strongly continuous semigroup in the same way as the $L^2(M, \mu)$ setting. The subordinated semigroup on $C_{\infty}(M)$ is denoted by $\{P_t^{(\alpha)}\}_{t\ge 0}$, with generator $	-(-\Delta)^{\alpha}$ (on $C_{\infty}(M)$).  
	
	Assume that the fractional Laplacian $	(-\Delta)^{\alpha}$ (on $L^2(M, \mu)$) is transient. The corresponding fractional Green operator $G^{(\alpha)}$ on $C_{\infty}(M)$ can be understood in the following way:
	\[G^{(\alpha)} f=\lim_{t\rightarrow +\infty}\int_{0}^{t}P_s^{(\alpha)} f ds,\]
	for $f\in C_{\infty}(M)$ such that the limit exists pointwise. 
	Note that $G^{(\alpha)}$ is given by the integral kernel $G^{(\alpha)}(\cdot, \cdot)$ as introduced in Definition \ref{defi-Green-alpha}. The following fact plays the role of fundamental theorem of calculus in semigroup theory.
	\begin{lemma}\cite[Lemma 1.3]{semigroup} \label{lem-generator}\rm
 Let $(A, \mathcal{D}(A))$ be the generator of a strongly continuous semigroup $\{T_t\}_{t\ge 0}$ on a Banach space $Y$. Then the following properties hold.
\begin{enumerate}
\item If $f\in \mathcal{D}(A)$, then $T_t f\in \mathcal{D}(A)$ and \[\frac{d}{dt}T_t f =T_t A f=AT_t f\] for all $t\ge 0$.
\item For each $t\ge 0$ and $f\in Y$, \[\int_{0}^{t} T_s f ds \in \mathcal{D}(A).\]
\item For each $t\ge 0$ and $f\in Y$,
\[T_t f - f=A\int_{0}^{t} T_s f ds.\]
\item For each $t\ge 0$ and $f\in \mathcal{D}(A)$,
\[T_t f - f=\int_{0}^{t} T_s A f ds.\]
\end{enumerate}
	\end{lemma}
	
	Now we assume that $G^{(\alpha)}$ is quasi-metric and there is a positive solution to (\ref{eq-non-local-integral-inequality-mu}). By the proof of Theorem \ref{thm-equivalence} above, there is a non-negative function $v\in \mathcal{F}^{(\alpha)}_e$, such that $v\le cm$, and
\begin{align}\label{eq-v-solution}
v=G^{(\alpha)} (v^q)+G^{(\alpha)} (\eta)
\end{align}
	holds. Here $\eta \in C_c(M)$  is a non-negative, non-zero function with $\eta \le c' m^q$. The following lemma gives a way to regularize each term.
	
	\begin{lemma}\label{lem-regular} \rm
		Let $w\in C_{\infty}(M)$, and let $f$ be a non-negative measurable function such that $f\le w$ $\mu$-a.e.. Then \[\int_{1}^{2} P_t f  dt\in \mathcal{D}(\Delta)\cap C^{\infty}(M).\]
		Here $\mathcal{D}(\Delta)$ is the domain of the generator $\Delta$ on $C_{\infty}(M)$.
	\end{lemma}
	\begin{proof}
		First note that
		\[\int_{1}^{2} P_t f  dt=P_{1/2}\left(\int_{1/2}^{3/2} P_t f  dt\right).\]
		Since $P_t$ is contractive with respect to $\norm{\cdot}_{\infty}$, we have $\int_{1/2}^{3/2} P_t f  dt \in L^{\infty}(M, \mu)$ and hence $\int_{1}^{2} P_t f  dt$ is smooth (see \cite[Theorem 7.15]{G06}). 
		
		From the Feller property we deduce that $P_t w\in C_{\infty}(M)$ for each $t\ge 0$. Since $0\le f\le w$, we have
		$P_{1/2} f \le P_{1/2} w$ and $P_{1/2} f$ is smooth. In particular, it follows that $P_{1/2} f\in C_{\infty}(M)$. We then have
		\[\int_{1}^{2} P_t f dt=\int_{1/2}^{3/2} P_t (P_{1/2} f )dt \in \mathcal{D}(\Delta)\]
		(cf. Lemma \ref{lem-generator} (2)).
	\end{proof}
	
	The following is a direct consequence of Definition \ref{defi-Green-alpha} and Fubini's theorem.			
	\begin{lemma} \rm
		For each non-negative measurable function $f$, and each $s>t\ge 0$,
		\[\int_{t}^{s}P_r(G^{(\alpha)}(f))dr=G^{(\alpha)}(\int_{t}^{s}P_r f dr).\]
	\end{lemma}
	
	%
	
By the assumption on  asymptotic decay of $G^{(\alpha)}$, we have \[\lim_{d(x,o)\rightarrow +\infty}m(x)=0.\]
	Then we can find some $w\in C_{\infty}(M)$ such that $w\ge m$.  
	For some $C>0$, $Cw$ dominates $v, v^q, \eta, G^{(\alpha)} (v^q), G^{(\alpha)} (\eta)$. It follows from Lemma \ref{lem-regular} that 
	\[\int_{1}^{2} P_t (v)  dt,~~ \int_{1}^{2} P_t (v^q)  dt, ~~\int_{1}^{2} P_t ( \eta)  dt \]
	and 	\[ G^{(\alpha)}\left(\int_{1}^{2} P_t  (v^q) dt\right)=\int_{1}^{2} P_t (G^{(\alpha)} (v^q))  dt, ~~G^{(\alpha)} \left(\int_{1}^{2} P_t ( \eta)  dt\right)=\int_{1}^{2} P_t ( G^{(\alpha)} (\eta))  dt \]
	are all in $\mathcal{D}(\Delta)\cap C^{\infty}(M)$.
	
	By \cite[Theorem 32.1]{Sato}, $\mathcal{D}(\Delta)\subseteq \mathcal{D}((-\Delta)^{\alpha})$, both of which are understood as domains of generators on $C_{\infty}(M)$.
	
	\begin{lemma} \rm
		Let $0\le f\in \mathcal{D}((-\Delta)^{\alpha})$, understood on $C_{\infty}(M)$. Suppose that  $G^{(\alpha)} f\in \mathcal{D}((-\Delta)^{\alpha})$. Then \[(-\Delta)^{\alpha}\left(G^{(\alpha)} f \right)=f.\] 
	\end{lemma}
	\begin{proof}
		Since $G^{(\alpha)} f\in \mathcal{D}((-\Delta)^{\alpha})$, we can compute \begin{align*}
			(-\Delta)^{\alpha}\left(G^{(\alpha)} f \right)&=\lim_{t \rightarrow 0} \frac{G^{(\alpha)} f- P_t^{(\alpha)}(G^{(\alpha)} f)}{t}\\&=\lim_{t \rightarrow 0} \frac{ \int_{0}^{\infty}P_s^{(\alpha)}f ds - \int_{0}^{\infty}P_{t+s}^{(\alpha)}f ds}{t}\\&=\lim_{t \rightarrow 0} \frac{ \int_{0}^{t}P_s^{(\alpha)}f ds}{t}\\&=f,
		\end{align*}
		where each limit is taken pointwise.
	\end{proof}
	\begin{proof}[Proof of Theorem \ref{thm-strong-solution}]
		
	``only if'' part: We apply the operator $(-\Delta)^{\alpha}\int_{1}^{2} P_t (\cdot)  dt$ to  each term of (\ref{eq-v-solution}) and obtain
\[(-\Delta)^{\alpha}\int_{1}^{2} P_t (v)  dt= (-\Delta)^{\alpha}\int_{1}^{2} P_t (G^{(\alpha)} (v^q))  dt+ (-\Delta)^{\alpha}\int_{1}^{2} P_t (G^{(\alpha)} (\eta))  dt.\]
By the above discussions, this can be simplified as
\[(-\Delta)^{\alpha}\int_{1}^{2} P_t (v)  dt=\int_{1}^{2} P_t (v^q)  dt+\int_{1}^{2} P_t ( \eta)  dt.\]

Note that
\[\int_{1}^{2} P_t (v^q)  dt\ge \left(\int_{1}^{2} P_t (v)  dt\right)^q.\]
We conclude that the positive function $h=\int_{1}^{2} P_t (v)  dt\in \mathcal{D}((-\Delta)^{\alpha})\cap C^{\infty}(M)$ and satisfies
\[(-\Delta)^{\alpha}h\ge h^q.\]

	``if'' part: Let $h$ be a positive smooth function on $M$ that is in the domain of $(-\Delta)^{\alpha}$ (on $C_{\infty}(M)$) and satisfies
	\[(-\Delta)^{\alpha}h\ge h^q.\]
	It follows that 
	\[G^{(\alpha)}\left((-\Delta)^{\alpha}h\right)\ge G^{(\alpha)}(h^q).\]
	
	Note that both $h$ and $(-\Delta)^{\alpha}h$ are positive. Hence we have \begin{align*}
	G^{(\alpha)}\left((-\Delta)^{\alpha}h\right)&=\lim_{t\rightarrow +\infty}\int_{0}^{t}P_s^{(\alpha)} \left((-\Delta)^{\alpha}h\right) ds\\&=\lim_{t\rightarrow +\infty} \left(h-P_s^{(\alpha)}h\right)\\&\le h,
	\end{align*}
where for the last equality we applied that $h\in \mathcal{D}((-\Delta)^{\alpha})$ (on $C_{\infty}(M)$).
We then obtain
\[h\ge G^{(\alpha)}(h^q),\]
which finishes the proof.
	\end{proof}

\appendix
\section{Some proofs for Section \ref{sect-potential}}\label{appendix-potential}
Here we give the detailed proofs for some results in Section \ref{sect-potential}, for the adaption to the current setting. No originality is claimed.
\begin{proof}[Proof of Lemma \ref{lem-kernel-limit}]
	Let $\nu \in \mathcal{M}^{+}(X)$ be concentrated on a Borel set $A\subset X$.
	
	If $K_{\nu}1\le 1$ in $A$, then for each $n\ge 1$, $K_{n,\nu}1 \le K_{\nu}1\le 1$ in $A$. By the weak maximum principle for $K_n$, we then have
	$K_{n,\nu}1\le b$ in $X$.
	
	By the monotone convergence theorem, $K_{\nu}1 \le b$ in $X$. The assertion follows.
\end{proof}

\begin{proof}[Proof of Lemma \ref{lem-kernel-multiple}]
	We first assume that $\eta$ is bounded.
	Suppose $\nu\in\mathcal{M}^{+}(X)$ is concentrated on a Borel set $A$. Set $\nu'=\eta\nu\in \mathcal{M}^{+}(X)$. Then $\nu' $ is also concentrated on $A$.
	
	If $\tilde{K}_{\nu} 1 \le 1$ in $A$, noticing that $\tilde{K}_{\nu}1 =K_{\nu'}1$, we have
	\[K_{\nu'}1\le 1 \text{~~in~~}A.\]
	Hence by the weak maximum principle for $K$, we obtain
	\[\tilde{K}_{\nu}1 =K_{\nu'}1 \le b \text{~~in~~} X.\]
	
	For a general non-negative measurable function $\eta$, consider $\eta_n=\eta \wedge n$ and set $\tilde{K}_n$ by $\tilde{K}_n(x,y)=K(x,y)\eta_n(y)$ for $n\ge 1$. By the bounded case, for each $n\ge 1$, $\tilde{K}_n$ satisfies the weak maximum principle with constant $b\ge 1$.
	The assertion follows by taking limit as is guaranteed by Lemma \ref{lem-kernel-limit}.
	%
	%
\end{proof}

\begin{proof}[Proof of Corollary \ref{cor-max-f}]
	Let $\tilde{K}(x,y)=K(x,y)f(y)$. By Lemma \ref{lem-kernel-multiple}, $\tilde{K}$ satisfies the weak maximum principle with constant $b$.
	
	Set $\nu'=1_{\{f>0\}}\nu\in  \mathcal{M}^{+}(X)$. Note that \[	K_{\nu}f=\tilde{K}_{\nu} 1=\tilde{K}_{\nu'}1,\]
	and that $\nu'$ is concentrated on $\{f>0\}\cap A$.
	The assertion follows by applying the weak maximum principle to $\tilde{K}$ and $\nu'$.
\end{proof}

\begin{proof}[Proof of Lemma \ref{int-type-ineq}]
	For any $y\in X$, set
	\begin{eqnarray*}
		E_y=\left\{z\in X: K_{\nu}1(z)\leq K_{\nu}1(y)\right\}.
	\end{eqnarray*}
	
	Clearly
	\begin{eqnarray*}
		K_{\nu}1_{E_y}(z)\leq K_{\nu}1(z)\leq K_{\nu}1(y)\quad\mbox{for all $z\in E_y$}.
	\end{eqnarray*}
	By Corollary \ref{cor-max-f}, 
	we obtain
	\begin{eqnarray*}
		K_{\nu}1_{E_y}(z)\leq bK_{\nu}1(y)\quad\mbox{for all $z\in X$}.
	\end{eqnarray*}
	
	Applying Lemma \ref{rearrange} to the $\sigma$-finite measure $\omega=K(x,\cdot)\nu$ and the function
	$f=K_{\nu}1$, and noticing that $\omega(E_y)=K_{\nu}1_{E_y}(x)$, we obtain
	\begin{eqnarray*}
		\int_0^{K_{\nu}1(x)}\varphi(t)dt&\leq& \int_{X}\varphi\left(\omega\left\{z\in X: K_{\nu}1(z)\leq K_{\nu}1(y)\right\}\right)\omega(dy) \\
		&=&\int_{X}\varphi(K_{\nu}1_{E_y}(x))\omega(dy) \\
		&\leq&\int_{X}\varphi(bK_{\nu}1(y))\omega(dy)=K_{\nu}[\varphi(b K_{\nu}1)](x).
	\end{eqnarray*}
	which completes the proof.
\end{proof}
%
\begin{proof}[Proof of Lemma \ref{lem-iter}]
	We argue by induction.
	When $k=0$, $\psi_0(f_0(x))=K_{\nu}1(x)= f_0(x)$ for all $x\in X$.
	
	For $k=1$,\[\psi_1(f_0(x))=\int_0^{K_{\nu}1(x)}\psi(s)ds.\] 
	
	Then (\ref{est-it}) follows from  Lemma \ref{int-type-ineq}
	by replacing $\varphi$ with $\psi$
	\begin{eqnarray}
		\psi_1(f_0(x))=\int_0^{K_{\nu}1(x)}\psi(s)ds\leq K_{\nu}(\psi(b K_{\nu}1))(x)=K_{\nu}(\varphi(K_{\nu}1))(x)=f_1(x),
	\end{eqnarray}
	for each $x\in X$ such that $K(x,\cdot)$ is $\sigma$-finite with respect to $\nu$.
	
	For the induction step, now we assume that the assertion holds for some $k\ge 1$.	
	
	For $y\in X$, set
	\begin{eqnarray*}
		E_y=\{z\in \Omega: f_k(z)\leq f_k(y)\}.
	\end{eqnarray*}
	Then we have by Lemma \ref{lem-handy},
	\begin{eqnarray}\label{G-new}
		K_{\nu}(1_{E_y}\varphi(f_{k-1}))(x)\leq 	b K_{\nu}(\varphi(f_{k-1}))(y)= bf_k(y)\quad\mbox{for all $x, y\in X$}.
	\end{eqnarray}
	Consider now an auxiliary kernel defined by
	\[\hat{K}(x,z)=K(x,z)1_{E_y}(z),\]
	and define a sequence $\{\hat{f}_l\}$ of functions similarly to (\ref{it-f}):
	\begin{eqnarray*}
		\hat{f}_0=\hat{K}_{\nu}1=K_{\nu}1_{E_y},\quad\hat{f}_{l+1}=\hat{K}_{\nu}(\varphi(\hat{f}_l))=K_{\nu}\left(1_{E_y}\varphi(\hat{f}_l)\right).
	\end{eqnarray*}
	It follows from (\ref{G-new}) that
	\begin{eqnarray*}
		\hat{f}_k(x)\le 	K_{\nu}(1_{E_y}\varphi(f_{k-1}))(x)\leq bf_k(y)\quad\mbox{for all $x,y\in X$}.
	\end{eqnarray*}
	By the inductive hypothesis, we have,  for each $x\in X$ such that $\hat{K}(x,\cdot)=K(x,\cdot)1_{E_y}(\cdot)$ is $\sigma$-finite with respect to $\nu$,
	\begin{eqnarray*}
		\psi_k(\hat{f}_0(x))\leq \hat{f}_k(x).
	\end{eqnarray*}
	In particular, the above estimate is true for each $x\in X$ such that $K(x,\cdot)$ is $\sigma$-finite with respect to $\nu$.
	It follows that
	\begin{eqnarray*}
		\psi\circ\psi_k(\hat{f}_0(x))\leq\psi(\hat{f}_k(x))\leq\psi(b f_k(y))=\varphi(f_k(y)),
	\end{eqnarray*}
	that is
	\begin{eqnarray}\label{G-hat-2}
		\psi\circ\psi_k(K_{\nu}1_{E_y})(x)\leq\varphi(f_k(y)),
	\end{eqnarray}
	for each $y\in X$, and each $x\in X$ such that $K(x,\cdot)$ is $\sigma$-finite with respect to $\nu$.
	
	Fix now some $x\in X$ such that $K(x,\cdot)$ is $\sigma$-finite with respect to $\nu$. Apply Lemma \ref{rearrange} with the $\sigma\text{-}$finite measure $\omega=K(x, \cdot)\nu$. We obtain,
	using (\ref{psi-k}) and (\ref{G-hat-2}), that
	\begin{align}
		\psi_{k+1}(f_0(x))&=\int_0^{f_0(x)}\psi\circ\psi_k(s)ds=\int_0^{\omega(X)}\psi\circ\psi_k(s)ds\nonumber\\
		&\leq\int_{X}\psi\circ\psi_k(\omega(z\in X:f_k(z)\leq f_k(y)))\omega(dy)\nonumber\\
		&=\int_{X}\psi\circ\psi_k(K_{\nu}1_{E_y}(x))\omega(dy)\nonumber\\
		&\leq\int_{X}\varphi(f_k(y))\omega(dy)\nonumber\\
		&=K_{\nu}(\varphi(f_k))(x)=f_{k+1}(x),
	\end{align}
	which finishes the inductive step.
\end{proof}

\begin{proof}[Proof of Theorem \ref{thm-iteration}]
	It is clear that for each $x\in A$,  $K(x,\cdot)$ is $\sigma$-finite with respect to $\nu$.
	
	Set $\varphi(t)=g(t+1)$ for $t\ge 0$. Let $\psi(t)=\varphi(b^{-1} t)$. Define $\psi_{k}$ and $f_k$ for $k\in \mathbb{N}$ in the same way as in Lemma \ref{lem-iter}.
	
	We claim that for each $x\in A$,
	\[u(x)\ge f_k(x) +1, \quad \forall k\in \mathbb{N}.\]
	Indeed, by assumption $u(x)\ge 1$ for each $x\in A$. It follows that for each $x\in A$,
	\[u(x)\ge K_{\nu}(g(u))(x)+1\ge K_{\nu}1(x)+1=f_0(x)+1.\]
	Suppose $u\ge f_k+1$ on $A$ is known. Then we have  for each $x\in A$,
	\[u(x)\ge K_{\nu}(g(u))(x)+1\ge K_{\nu}(\varphi(f_k))(x)+1=f_{k+1}(x)+1.\]
	Hence the claim holds.
	
	By Lemma \ref{lem-iter},  for each $x\in A$, $\forall k\in \mathbb{N}$,
	\[u(x)\ge f_k(x) +1\ge \psi_{k}(K_{\nu}1)(x)+1.\]
	
	Note that $\psi \ge 1$ and $\psi$ is non-decreasing. It is then clear that
	\[	\psi_{1}(t)=\int_0^t\psi\circ\psi_0(s)ds=\int_0^t\psi(s)ds\ge t=\psi_0(t).\]	By induction we have
	$\forall k\ge 1$,
	\[\psi_{k+1}(t)=\int_{0}^{t}\psi\circ\psi_k(s)ds\ge \int_0^t\psi\circ\psi_{k-1}(s)ds =\psi_{k}(t).\]
	
	Set $\psi_{\infty}=\lim_{k\rightarrow +\infty}\psi_{k}$. Then $\psi_{\infty}$ is non-decreasing and we have
	\[u\ge \psi_{\infty}(K_{\nu}1)+1\]
	holds on $A$.
	
	Fix $x\in A$. Let $t_0=K_{\nu}1(x)$. Then $t_0<u(x)<+\infty$.
	Without loss of generality, we assume that $t_0>0$. Then $\psi_{\infty}$ is finite on $[0, t_0]$ and is positive on $(0, t_0]$. For each $t\in [0, t_0] $, we have
	\begin{align*}
		\int_0^t\psi\circ\psi_{\infty}(s)ds &=\int_0^t\psi(\lim_{k\rightarrow +\infty}\psi_k(s))ds\\&=\int_0^t\lim_{k\rightarrow +\infty}\psi(\psi_k(s))ds\\&=\lim_{k\rightarrow +\infty}\int_0^t\psi(\psi_k(s))ds\\&=\psi_{\infty}(t).
	\end{align*}
	Hence on $[0, t_0]$,
	\begin{equation}\label{ode-1}
		\frac{d}{d t}\psi_{\infty}(t)=\psi\circ\psi_{\infty}(t).
	\end{equation}
	
	Consider
	\[\Psi(r):=\int_{0}^{r} \frac{ds}{\psi(s)}=\int_{0}^{r} \frac{ds}{g(b^{-1} s +1)}=bF(1+r/b).\]
	Combining with  \eqref{ode-1}, we derive that
	$$\frac{d\Psi(\psi_{\infty}(t))}{dt}=1,$$
	and thus $\Psi(\psi_{\infty}(t))=t$. Hence
	\[t_0=\Psi(\psi_{\infty}(t_0))<\lim_{r\rightarrow +\infty}\Psi(r)=b\int_{1}^{+\infty}
	\frac{d t}{g(t)}.\]
\end{proof}

\begin{proof}[Proof of Theorem \ref{thm-iteration-h}]
	Consider a new kernel $\hat{K}(x,y)=\frac{1}{h(x)}K(x,y) h(y)^q$.  In the same way as in the proof of Theorem \ref{thm-iteration}, we see that $K(x,\cdot)$ is $\sigma$-finite with respect to $\nu$ at each $x\in A$. Since $h$ is positive and bounded, we see that $\hat{K}(x, \cdot)$ is $\sigma$-finite with respect to $\nu$ at each $x\in A$.
	
	Note that $\hat{K}(x,y)=\tilde{K}(x,y)h^{q+1}(y)$. By Lemma \ref{lem-kernel-multiple}, $\hat{K}$ also satisfies the weak maximum principle with constant $b$.
	
	Let $v=\frac{u}{h}$. The inequality
	\[u\ge K_{\nu}(u^q)+h\]
	is equivalent to
	\[v\ge \hat{K}_{\nu}(v^q)+1.\]
	
	By Corollary \ref{coro-iteration-1}, at each $x\in A$,
		\[\hat{K}_{\nu} 1(x)<\frac{b}{q-1},\]
		which is equivalent to
	\[K_{\nu}(h^q)(x)<\frac{b}{q-1} h(x).\]	
\end{proof}
\begin{proof}[Proof of Lemma \ref{lem-Ptolemy}]
	Without loss of generality, we can assume that \[K(x,y)=\max\{K(x,y), K(y,z),K(o,z), K(o,x)\}.\]
	Then by the quasi-metric property,
	\begin{align*}
		\kappa^2 (K(x,z)K(o,y))&\ge (K(x,y)\wedge K(y,z))\cdot (K(x,y)\wedge K(x,o))\\
		&=K(y,z)K(x,o)\\
		&\ge (K(x,y)K(o,z))\wedge(K(y,z)K(o,x)).
	\end{align*}
\end{proof}

\begin{proof}[Proof of Lemma \ref{lem-quasi-weak}]
	For any constant $N\ge 1$, define $K_N(\cdot, \cdot)= K(\cdot, \cdot)\wedge N$.	Now we show that $K_N$ is a quasi-metric kernel. Indeed, symmetry is clear. And we have for all $x,y,z\in X$,
	\begin{align*}
		K_N(x,y)\wedge K_N(y,z)&=\min\{K(x,y), K(y,z), N\}\\
		&=\min\{K(x,y), K(y,z)\}\wedge N\\
		&\le (\kappa K(x,z))\wedge N \\ &\le \kappa (K(x,z)\wedge N)\\
		&=\kappa K_{N}(x,z).
	\end{align*}
	
	Let $\nu\in \mathcal{M}^{+}(X)$ be concentrated on $A$.  Assume that $K_{\nu}1(x)\leq1$ for all $x\in A$, then $K_{N, \nu}1(x)\le 1$
		for all $x\in A$. Note that
			for each $x\in X$,
			\begin{equation}\label{eq-supp}
				K_{N, \nu}1(x)=\int_X K_N(x,y)d\nu(y)=\int_A K_N(x,y)d\nu(y).
	\end{equation}
	Fix an arbitrary $x\in A^c$ and let $0<\varepsilon<1$. Choose $y_{\varepsilon}\in A$ satisfying
	
	\[K_N(x, y_{\varepsilon})> (1-\varepsilon) \sup_{A} K_N(x, \cdot).\]
	Note that for each $y\in A$, $K_N(x,y)<\frac{1}{1-\varepsilon} K_N(x, y_{\varepsilon})$.
	Then by (\ref{eq-supp}), we obtain
		\begin{align*}
			K_{N, \nu}1(x)&=\int_A K_N(x,y)d\nu(y)\\
			&=\int_A \min\{ \frac{1}{1-\varepsilon} K_N(x, y_{\varepsilon}), K_N(x,y)\}d\nu(y)\\
			&\le \frac{1}{1-\varepsilon}\int_A \min\{ K_N(x, y_{\varepsilon}), K_N(x,y)\}d\nu(y)\\
			&\le  \frac{1}{1-\varepsilon}\int_A \kappa K_N(y_{\varepsilon},y)d\nu(y)\\
			&\le  \frac{\kappa}{1-\varepsilon}.
		\end{align*}
		Since $0<\varepsilon<1$ is arbitrary, we have
		\[K_{N, \nu}1(x)\le \kappa.\]
	
	By monotone convergence,  for each $x\in A^c$,
	{	\[K_{\nu}1(x)=\lim_{N\rightarrow +\infty} K_{N, \nu}1(x)\le \kappa.\]}
\end{proof}

\begin{proof}[Proof of Lemma \ref{lem-quasi-m}]
	The inequality (\ref{eq-tilde-K}) amounts to
	\[\min\{K(x,y)k(z), K(y,z)k(x)\}\le \kappa^2 K(x,z)k(y).\]
	
	By symmetry, it suffices to consider two cases:
	that
	\[K(x,y)=\max\{K(x,y), K(y,z),K(o,z), K(o,x)\}, \tag{1}\]
	or that
	\[K(o,x)=\max\{K(x,y), K(y,z),K(o,z), K(o,x)\}. \tag{2}\]
	
	For case (1), we have
	\begin{align*}
		\kappa^2 K(x,z) k(y)&\ge  \min\{K(x,y), K(y,z)\}\cdot ((\kappa K(o,y)) \wedge (\kappa c))\\
		&\ge  \min\{K(x,y), K(y,z)\}\cdot  \min\{K(o,x), K(x,y), c\}\\
		&=K(y,z)\cdot (K(o,x)\wedge c)\\
		&=K(y,z)k(x).
	\end{align*}
	
	Similarly, for case (2), we have
	\begin{align*}
		\kappa^2 K(x,z)k(y) &= 	\kappa K(x,z) \cdot  ((\kappa K(o,y)) \wedge (\kappa c))\\
		&\ge  	\kappa K(x,z) \cdot  \min\{K(o,x), K(x,y), c\}\\
		&= 	\kappa K(x,z)\cdot (K(x,y)\wedge c).
	\end{align*}
	Then we proceed by dividing it into two further cases.
	If $K(x,y)\le c$,  the above inequality leads to
	\begin{align*}
		\kappa^2 K(x,z) k(y)&\ge	\kappa K(x,z)\cdot (K(x,y)\wedge c)\\&\ge (K(o, x)\wedge K(o,z)) \cdot K(x,y)\\&
		=K(o,z) K(x,y)\\&\ge K(x,y)k(z).
	\end{align*}
	Otherwise $K(x,y)>c$, and we have
	\begin{align*}
		\kappa^2 K(x,z)k(y)&\ge	\kappa K(x,z)\cdot (K(x,y)\wedge c)\\
		&= c 	\kappa K(x,z)\\ &\ge c (K(x,y)\wedge K(y,z))\\ &\ge \min\{K(x,y)k(z), K(y,z)k(x)\}.
	\end{align*}
	
\end{proof}
\begin{remark}\rm
	Observe that for fixed $x,y,z\in X$, the inequality (\ref{eq-tilde-K}) only involves the restriction of $K$ on the set $\{o, x,y,z\}$. We can consider a modified kernel $\hat{K}$ on the space $\hat{X}=\{o, x,y,z\}$: $\hat{K}(u,v)=K(u, v)$ for $u, v\in \hat{X}\backslash\{o\}$, and $\hat{K}(u,v)=K(u, v)\wedge c$ if $u=o$ or $v=o$. Then $\hat{K}$ is a quasi-metric kernel on $\hat{X}$ and we can simply apply the Ptolemy inequality (\ref{eq-Ptolemy})  to  $\hat{K}$.
	
\end{remark}

\section{A brief summary of Dirichlet form theory}\label{appendix-DF}



Dirichlet form theory can be viewed as an organic generalization of the $L^2$-aspects of elliptic/parabolic PDEs, with deep connections to potential theory and Markov processes. For example, weak solutions can be formulated in terms of Dirichlet forms in a natural way. Here we present a minimal summary of basic notions and facts needed in Section \ref{sect-solutions} and Section \ref{sect-strong-solutions}. We refer to the standard references \cite{FOT, Chen-Fuku} for comprehensive accounts.

	Let $(X, d)$ be a locally compact separable metric space and $\mu$ be a Radon measure on $X$ with full support. A regular Dirichlet form $(\mathcal{E}, \mathcal{F})$ on $(X, d, \mu)$ consists of a dense subspace $\mathcal{F}$ of $L^2(X, \mu)$ and a non-negative definite symmetric bilinear form $\mathcal{E}: \mathcal{F}\times \mathcal{F}\rightarrow \mathbb{R}$ such that the following conditions hold:
	\begin{itemize}
		\item  $\mathcal{F}$ is complete with respect to the inner product $\mathcal{E}_1(\cdot, \cdot)\defeq\mathcal{E}(\cdot, \cdot)+(\cdot, \cdot)_{L^2(X,\mu)}$;
		\item $\mathcal{F}\cap C_c(X)$ is dense in both $\mathcal{F}$ (with $\mathcal{E}^{1/2}_1$-norm) and $C_c(X)$ (with $\sup$-norm);
		\item if $u\in \mathcal{F}$, then $ v\defeq(0\vee u)\wedge 1 \in \mathcal{F}$, and $\mathcal{E}(v, v)\le \mathcal{E}(u, u)$.
	\end{itemize}
A classical example of a Dirichlet form is the Sobolev space $W^{1,2}(\mathbb{R}^n)$ with the Dirichlet energy form on it:
\[\mathcal{E}(u, v)=\int_{\mathbb{R}^n}\nabla u\cdot\nabla v d\mu, ~~u, v\in \mathcal{F}=W^{1,2}(\mathbb{R}^n).\]

Let $\mathcal{D}$ be the space of  functions $u\in \mathcal{F}$ such that for some constant $C\ge 0$, $\forall v\in \mathcal{F}, ~\mathcal{E}(u, v)\le C\norm{v}_{L^2(X,\mu)}$. By the Riesz representation theorem, there exists a non-negative definite self-adjoint operator $L$ with domain $\mathcal{D}$ such that
for all $u\in \mathcal{D}$, and for all $v\in \mathcal{F}$
\[\mathcal{E}(u, v)=(Lu, v)_{L^2(X,\mu)}.\]
The operator $L$ is called the generator of the Dirichlet form $(\mathcal{E},\mathcal{F})$. For the above example of classical Dirichlet energy form on $\mathbb{R}^n$, the corresponding generator is $(L,  \mathcal{D})=(-\Delta, W^{2,2}(\mathbb{R}^n))$, where $-\Delta$ is the (non-negative signed) Laplace operator.

The operator $L$ naturally generates a strongly continuous contraction semigroup $\{T_t\}_{t\ge 0}$ on $L^2(X, \mu)$ where $T_t= \exp (-t L)$ is understood in the sense of spectral calculus. The semigroup  $\{T_t\}_{t\ge 0}$ is Markovian, that is, $0\le T_t u\le 1$ for each $u\in L^2(X, \mu)$ with $0\le u\le 1$. The Dirichlet form, its generator and the corresponding semigroup in fact determine each other uniquely, see  \cite{FOT, Chen-Fuku} for more details.

To allow more general weak type solutions, it is convenient to adopt the notion of extended Dirichlet space.	Given a regular Dirichlet form $(\mathcal{E}, \mathcal{F})$, a measurable real-valued function $u$  is said to belong to the extended Dirichlet space $\mathcal{F}_e$ if there exists an $
	\mathcal{E}$-Cauchy sequence $\{u_k\}_{k\ge 1}$, called an approximating sequence, such that $u_k\rightarrow u$ $\mu$-a.e. as $k\rightarrow\infty$. Clearly $\mathcal{F}_e$ is a linear space containing $\mathcal{F}$. The bilinear form $\mathcal{E}$ can be naturally extended to $\mathcal{F}_e$ by setting $\mathcal{E}(u, u)=\lim_{k\rightarrow \infty}\mathcal{E}(u_k, u_k)$ which is independent of the choice of the approximating sequence $\{u_k\}_{k\ge 1}$. In the Euclidean space setting, extended Dirichlet spaces are closely related to homogeneous Sobolev spaces and Riesz potential spaces. 

We are only concerned with irreducible Dirichlet forms, i.e. those with a certain ergodic property(see \cite[Definition 2.1.1]{Chen-Fuku}). There is a fundamental dichotomy, transience v.s. recurrence, for the long time asymptotic behaviour of an irreducible Dirichlet form. For $t>0$, consider  an operator $S_t$ defined as
\[S_t f = \int_{0}^{t} T_s f d s,\]
first for functions  $f\in L^1(X, \mu)\cap L^2(X, \mu)$ and then for functions belonging to  $L^1(X, \mu)$ by extension. Denote by $L^1_+(X,\mu)$ the set of non-negative functions in $L^1(X, \mu)$. By monotonicity, for  $f\in L^1_+(X, \mu)$,
the potential operator $Sf=\lim\limits_{N\rightarrow +\infty} S_N f$ can be defined. In fact, $Sf$ makes sense for each non-negative measurable function $f$ (see  \cite[Section 1.5]{FOT}). 

The Dirichlet form $(\mathcal{E},\mathcal{F})$
	(equivalently, the semigroup  $\{T_t\}_{t\ge 0}$, or the generator $L$) is said to be transient if for each $f\in L^1_+(X,\mu)$, $Sf(x)<+\infty$ for $\mu$-a.e. $x\in X$. An equivalent formulation is that the extended Dirichlet space $\mathcal{F}_e$ is a Hilbert space with respect to $\mathcal{E}(\cdot, \cdot)$(see \cite[Theorem 1.6.3]{FOT}). It is recurrent if for each $f\in L^1_+(X,\mu)$, $Sf(x)=0$ or $+\infty$ for $\mu$-a.e. $x\in X$.

For irreducible Dirichlet forms (see  \cite[Section 1.6]{FOT}), the transience property admits convenient equivalent criteria. The following fact is contained  in \cite[Lemma 1.6.3 and Lemma 1.6.4]{FOT}.

\begin{proposition}\label{prop-trans}\rm
		Let $(\mathcal{E},\mathcal{F})$ be a regular Dirichlet form. Suppose that it is irreducible. Then  $(\mathcal{E},\mathcal{F})$ is transient if and only if one of the following equivalent conditions holds.
		\begin{enumerate}
			\item for some $f\in L^1_+(X,\mu)$ with $\mu\left(\{x\in M: f(x)>0\}\right)>0$, $Sf(x)<+\infty$ for $\mu$-a.e. $x\in X$;
			\item for some non-negative measurable function $f$ with $\mu\left(\{x\in M: f(x)>0\}\right)>0$, $Sf(x)<+\infty$ for $\mu$-a.e. $x\in X$.
		\end{enumerate}
\end{proposition}
\begin{remark}\rm
In the cases we are concerned with, the heat semigroup $\{P_t\}_{t\ge 0}$ on a connected complete Riemannian manifold is always irreducible, as can be seen from the positivity of the heat kernel. As shown in \cite[Theorem 3.1]{Okura}, the fractional semigroups $\{P_t^{(\alpha)}\}_{t\ge 0}$ obtained from the subordination procedure is irreducible as well.
\end{remark}
The framework of  a transient regular Dirichlet form accommodates a natural treatment of weak type super-harmonic and harmonic functions with respect to the generator $L$. 

A non-negative measurable function is said to be of finite energy if  $\int_{X} f Sf d\mu <+\infty$. The following theorem is important when discussing solutions in the $\mathcal{F}_e$-sense in Section \ref{sect-solutions}.
\begin{theorem}\cite[Theorem 1.5.4]{FOT} \label{thm-df}\rm
	Let $(\mathcal{E}, \mathcal{F})$ be a transient regular Dirichlet form with the extended Dirichlet space $\mathcal{F}_e$. Let $f$ be a non-negative measurable function of finite energy. Then the following hold
	\begin{itemize}
		\item $Gf\in \mathcal{F}_e$;
		\item $f v\in L^1(X,\mu)$ for each $v\in \mathcal{F}_e$;
		\item $\mathcal{E}(Gf, v)=\int_{X} f v d \mu$  for each $v\in \mathcal{F}_e$.
	\end{itemize}
	
\end{theorem}
A  Radon measure  $\nu$ on $X$ is said to be of finite $0$-order energy integral if for some $C>0$, $\forall f\in \mathcal{F}\cap C_c(M)$, $\int_{M}|f|d \nu \le \sqrt{\mathcal{E}(f, f)}$ (see \cite[Section 2.2]{FOT}). The following can be viewed as an abstract criterion of weak type super-harmonic functions.
\begin{proposition}\cite[Lemma 2.2.10]{FOT} \label{prop-weak}\rm	Let $(\mathcal{E}, \mathcal{F})$ be a transient regular Dirichlet form with the extended Dirichlet space $\mathcal{F}_e$.
The following are equivalent for $v\in \mathcal{F}_e$:
\begin{enumerate}
\item there is a Radon measure  $\nu$ on $X$ of finite $0$-order energy integral, such that for each $\varphi\in \mathcal{F}_e$ with a quasi-continuous version $\tilde{\varphi}$,
\[\mathcal{E}(v, \varphi)=\int_M \tilde{\varphi} d \nu;\]
\item for each $\varphi\in \mathcal{F}_e$ with $\varphi\ge 0$ $\mu$-a.e., 
\[\mathcal{E}(v, \varphi)\ge 0;\]
\item for each $\varphi\in \mathcal{F}\cap C_c(X)$ with $\varphi\ge 0$, 
\[\mathcal{E}(v, \varphi)\ge 0.\]
\end{enumerate}
\end{proposition}





\textbf{Acknowledgements.} 
Prof. I. Verbitsky and Prof. A. Grigor'yan patiently answered many of our questions during the preparation of this work. We are grateful to their valuable help. Y. Sun would like to thank Prof. J. Xiao for helpful discussions and Prof. D. Ye for suggesting the possible application to the Hardy-H\'{e}non type equation (\ref{eq-Henon}). The literature \cite{BQ} is pointed out to us by Prof. X. Tang and Prof. Y. Li, to whom we are also grateful.


\end{document}